\documentclass[oneside,openany,a4paper]{amsart}
\usepackage[a4paper, margin=1in]{geometry}
\usepackage{amssymb}
\usepackage[centertags]{amsmath}
\usepackage{amsthm}
\usepackage{amsfonts}
\usepackage{eucal}
\usepackage{mathrsfs}
\usepackage{enumerate}

\usepackage[all,cmtip]{xy}
\usepackage{graphicx, eepic}

\usepackage{lmodern}
\usepackage[utf8]{inputenc}
\usepackage[T1]{fontenc}
\usepackage{pdfrender, xcolor}

\usepackage{pdfpages}
\usepackage{titletoc}
\usepackage{booktabs}
\usepackage[shortlabels]{enumitem}
\usepackage{mathtools}
\usepackage{thmtools}

\usepackage{epstopdf}
\usepackage{epsfig}
\usepackage{microtype}

\usepackage{hyperref}
\hypersetup{colorlinks}
\hypersetup{
    bookmarks=true,         
    unicode=false,          
    pdftoolbar=true,        
    pdfmenubar=true,        
    pdffitwindow=false,     
    pdfstartview={FitH},    
    pdftitle={My title},    
    pdfauthor={Author},     
    pdfsubject={Subject},   
    pdfcreator={Creator},   
    pdfproducer={Producer}, 
    pdfkeywords={keyword1} {key2} {key3}, 
    pdfnewwindow=true,      
    colorlinks=true,       
    linkcolor=red,          
    citecolor=magenta,        
    filecolor=violet,      
    urlcolor=blue      
}
\usepackage{multirow, url}
\usepackage{textcomp}
\DeclareMathAlphabet{\cmcal}{OMS}{cmsy}{m}{n}
\makeatletter
\def\blfootnote{\xdef\@thefnmark{}\@footnotetext}
\makeatother

\newtheoremstyle{thm}
  {3pt}
  {3pt}
  {\em}
  {0pt}
  {\bfseries}
  {}
  {5pt}
  {}
\newtheoremstyle{rem}
  {3pt}
  {3pt}
  {}
  {0pt}
  {\bfseries}
  {.}
  {5pt}
  {}

\newtheorem{thm}{Theorem}[section]

\newtheorem{lem}[thm]{Lemma}
\newtheorem{prop}[thm]{Proposition}
\newtheorem{defn}[thm]{Definition}

\newtheorem{conj}[thm]{Conjecture}

\theoremstyle{definition}

\theoremstyle{rem}
\newtheorem{rem}[thm]{{Remark}}

\newtheorem{theorem}{Theorem}[section]
\newtheorem{corollary}[thm]{Corollary}
\newtheorem{lemma}[thm]{Lemma}
\newtheorem{proposition}[thm]{Proposition}
\newtheorem{definition}[thm]{Definition}
\newtheorem{remark}[thm]{{Remark}}

\numberwithin{equation}{section} \numberwithin{table}{section}
	
\newtheorem*{thm*}{Theorem}
\newtheorem*{rem*}{Remark}
\newtheorem*{rems*}{Remarks}
\newtheorem*{exam*}{Example}
\newtheorem*{exams*}{Examples}

\makeatletter
\newcommand{\neutralize}[1]{\expandafter\let\csname c@#1\endcsname\count@}
\makeatother

\usepackage{enumerate}

\def\bos#1{{\mathbf{#1}}}

 \newcommand{\pa}[1]{\frac{\partial}{\partial{#1}}}
  \newcommand{\prt}[2]{\frac{\partial{#1}}{\partial{#2}}}

  \newcommand{\pr}{\operatorname{prim}}

  \newcommand{\nc}{\newcommand}
  \newcommand{\be}{\begin{eqnarray*}}
  \newcommand{\ee}{\end{eqnarray*}}
  \newcommand{\bea}{\begin{eqnarray}}
  \newcommand{\eea}{\end{eqnarray}}

   \nc{\bei}{\begin{itemize}}
   \nc{\eei}{\end{itemize}}
   \nc{\bee}{\begin{enumerate}}
   \nc{\eee}{\end{enumerate}}
   \nc{\bet}{\begin{theorem}}
   \nc{\eet}{\end{theorem}}
   \nc{\bed}{\begin{definition}}
   \nc{\eed}{\end{definition}}
   \nc{\bel}{\begin{lemma}}
   \nc{\eel}{\end{lemma}}
   \nc{\bep}{\begin{proposition}}
   \nc{\eep}{\end{proposition}}
   \nc{\bec}{\begin{corollary}}
   \nc{\eec}{\end{corollary}}
   \nc{\ber}{\begin{remark}}
   \nc{\eer}{\end{remark}}
   \nc{\beex}{\begin{example}}
   \nc{\eeex}{\end{example}}

   \nc{\bpm}{\begin{pmatrix}}
   \nc{\epm}{\end{pmatrix}}
   \nc{\bspm}{\left(\begin{smallmatrix}}
   \nc{\espm}{\end{smallmatrix}\right)}

\newcommand{\cA}{\mathcal{A}}

\newcommand{\cC}{\mathcal{C}}

\newcommand{\cE}{\mathcal{E}}
\newcommand{\cF}{\mathcal{F}}

\newcommand{\cM}{\mathcal{M}}

\newcommand{\cO}{\mathcal{O}}

\newcommand{\cR}{\mathcal{R}}

\newcommand{\cX}{\mathcal{X}}

\newcommand{\bA}{\mathbb{A}}

\newcommand{\bC}{\mathbb{C}}

\newcommand{\bF}{\mathbb{F}}

\newcommand{\bN}{\mathbb{N}}

\newcommand{\bQ}{\mathbb{Q}}
\newcommand{\bR}{\mathbb{R}}

\newcommand{\bZ}{\mathbb{Z}}

\newcommand{\BP}{\mathbf{P}}

\newcommand{\fra}{\mathfrak a}

\nc{\frf}{\mathfrak{f}} 

\newcommand{\frg}{\mathfrak g}

\nc{\frs}{\mathfrak{s}}  
\nc{\frt}{\mathfrak{t}} 
\nc{\fru}{\mathfrak{u}}
  
\nc{\lsl}{\mathfrak{sl}}
\nc{\lgl}{\mathfrak{gl}}

\nc{\upsi}{\underline{\psi}}
\nc{\uchi}{\underline{\chi}}

\DeclareMathOperator{\Tr}{Tr}

\DeclareMathOperator{\Hom}{Hom}

\DeclareMathOperator{\val}{val}

\DeclareMathOperator{\Res}{Res}

\DeclareMathOperator{\id}{id}

\DeclareMathOperator{\End}{End}

\newcommand{\lra}{\longrightarrow}    

\nc{\surjto}{\twoheadrightarrow}
\nc{\ts}{\times}
\nc{\ds}{\displaystyle}
\nc{\nd}{\noindent}  
\nc{\ud}{\underline}
\nc{\ov}{\overline}
\nc{\maplra}[1]{\buildrel #1 \over \lra}
\nc{\mapto}[1]{\buildrel #1 \over \to}
\nc{\setb}[1]{\{  #1\}}

 \nc{\cHom}{\mathcal{H}om}

\nc{\cdruur}[8] {\begin{CD} 
#1 @>#2>> #3\\ 
@AA#4A @AA#5A\\ 
#6 @>#7>> #8 
\end{CD} }
\nc{\cdrddr}[8] {\begin{CD} 
#1 @>#2>> #3\\ 
@VV#4V @VV#5V\\ 
#6 @>#7>> #8 
\end{CD} }

\nc{\dia}[8]{\xymatrix{ 
&#1 \ar@{-}[ld]_{#2} \ar@{-}[rd]^{#3} \\
#4 \ar@{-}[rd]_{#6} & &#5 \ar@{-}[ld]^{#7}\\ 
&#8} }

\nc{\diam}[9]{\xymatrix{ 
&#1 \ar@{-}[ld]_{#2}  \ar@{-}[d]^{#3} \ar@{-}[rd]^{#4} \\
#5 \ar@{-}[rd]_{#8}     & #6 \ar@{-}[d]_{#9}      & #7   \ar@{-}[ld]^{2} \\
& \bQ} } 

\nc{\sumn}[2][n]{#2_{1} +#2_{2}+ \cdots + #2_{#1}}
\nc{\poly}[3][n]{#2_{#1}#3^{#1} +#2_{#1-1}#3^{#1-1}  \cdots + #2_{1} #3+ #2_0}
\nc{\dpoly}[3][n]{#1#2_{#1}#3^{#1-1} +(#1-1)#2_{#1-1}#3^{#1-1}  \cdots +2 #2_{2} #3+ #2_1}
\nc{\mpoly}[3][n]{#3^{#1} +#2_{#1-1}#3^{#1-1}  \cdots + #2_{1} #3+ #2_0}

\nc{\vpar}[4]{    \left \{ \begin{array}{cc} #1 & \textrm{if } #2, \\
&\\
#3 & \textrm{if } #4. 
\end{array}\right. }

\nc{\vparr}[4]{    \left \{ \begin{array}{cc} #1 & \textrm{if } #2, \\
&\\
#3 & \textrm{if } #4, 
\end{array}\right. }

\nc{\ary}[5]{#1: \left\{ \begin{array}{ll} #2 &\mapsto #3 \\ #4 &\mapsto #5 \end{array} \right.}
 \nc{\bedm}{\begin{displaymath}}
 \nc{\eedm}{\end{displaymath}}
 \nc{\art}{\hbox{\bf Art}^\Z}
 \nc{\bvx}{\bos{B\!\!V}_{\! \!X}}

\newcommand{\pmat}{\left(\begin{matrix}}   
\newcommand{\epmat}{\end{matrix}\right)}   
\newcommand{\psmat}{\left(\begin{smallmatrix}}    
\newcommand{\epsmat}{\end{smallmatrix}\right)}
\nc{\twotwo}[4]{\pmat #1 & #2 \\ #3 & #4 \epmat}
\nc{\thrthr}[9]{\pmat #1 & #2 & #3 \\ #4 & #5 & #6 \\ #7 & #8 & #9 \epmat}
\nc{\stwotwo}[4]{\psmat #1 & #2 \\ #3 & #4 \epsmat}
\nc{\sthrthr}[9]{\psmat #1 & #2 & #3 \\ #4 & #5 & #6 \\ #7 & #8 & #9 \epsmat}

\def\eqalign#1{\null\,\vcenter{\openup\jot\m@th
\ialign{\strut\hfil$\displaystyle{##}$&$\displaystyle{{}##}$\hfil
\crcr#1\crcr}}\,}

\def\eqn#1#2{
\xdef #1{(\nsecsym\the\meqno)}
\global\advance\meqno by1
$$#2\eqno#1\eqlabeL#1
$$}

\def\b{\beta}

\def\g{\gamma}  \def\G{\Gamma}

\def\l{\lambda}  
\def\m{\mu}

\def\o{\omega}  \def\O{\Omega}

\def\cA{{\mathcal A}}

\def\Z{\mathbb{Z}}

\def\ma{\mathfrak{a}}

\def\mm{\mathfrak{m}}

\def\rd{\partial}

\def\wt{\hbox{\it wt}}
\def\ch{\hbox{\it ch}}

\begin{document}


\catcode`\@=11 

\global\newcount\nsecno \global\nsecno=0
\global\newcount\meqno \global\meqno=1
\def\newsec#1{\global\advance\nsecno by1
\eqnres@t
\section{#1}}
\def\eqnres@t{\xdef\nsecsym{\the\nsecno.}\global\meqno=1}
\def\sequentialequations{\def\eqnres@t{\bigbreak}}\xdef\nsecsym{}

\def\draftmode{\message{ DRAFTMODE }

{\count255=\time\divide\count255 by 60 \xdef\hourmin{\number\count255}
\multiply\count255 by-60\advance\count255 by\time
\xdef\hourmin{\hourmin:\ifnum\count255<10 0\fi\the\count255}}}
\def\nolabels{\def\wrlabeL##1{}\def\eqlabeL##1{}\def\reflabeL##1{}}
\def\writelabels{\def\wrlabeL##1{\leavevmode\vadjust{\rlap{\smash%
{\line{{\escapechar=` \hfill\rlap{\tt\hskip.03in\string##1}}}}}}}%
\def\eqlabeL##1{{\escapechar-1\rlap{\tt\hskip.05in\string##1}}}%
\def\reflabeL##1{\noexpand\llap{\noexpand\sevenrm\string\string\string##1}
}}

\nolabels

\def\eqn#1#2{
\xdef #1{(\nsecsym\the\meqno)}
\global\advance\meqno by1
$$#2\eqno#1\eqlabeL#1
$$}

\def\eqalign#1{\null\,\vcenter{\openup\jot\m@th
\ialign{\strut\hfil$\displaystyle{##}$&$\displaystyle{{}##}$\hfil
\crcr#1\crcr}}\,}

\def\ket#1{\left|\bos{ #1}\right>}\vspace{.2in}
   \def\bra#1{\left<\bos{ #1}\right|}
\def\oket#1{\left.\bos{ #1}\right>}
\def\obra#1{\left<\bos{ #1}\right.}
\def\epv#1#2#3{\left<\bos{#1}\left|\bos{#2}\right|\bos{#3}\right>}
\def\qbvk#1#2{\bos{\left(\bos{#1},\bos{#2}\right)}}
\def\Hoch{{\tt Hoch}}
\def\rrd{\up{\rightarrow}{\rd}}
\def\lrd{\up{\leftarrow}{\rd}}
   \nc{\hr}{[\![\hbar]\!]}
   \nc{ \cAb}{\cA\!(b)}
   \nc{\bNn}{\bZ_{\geq 0}}
   \nc{\Ab}{A\!(b)}
   \nc{\modulo}{\operatorname{mod}}
   
\def\foot#1{\footnote{#1}}

\catcode`\@=12 

\def\fr#1#2{{\textstyle{#1\over#2}}}
\def\Fr#1#2{{#1\over#2}}
\def\ato#1{{\buildrel #1\over\longrightarrow}}

\newcommand{\QFT}{\operatorname{QFT}}

\draftmode

\title{Batalin-Vilkovisky formalism in the $p$-adic Dwork theory}

\author{Dohyeong Kim}
\email{dohyeongkim@snu.ac.kr}

\author{Jeehoon Park} 
\email{jeehoonpark@postech.ac.kr}

\author{Junyeong Park}
\email{junyeongp@gmail.com}


%

\date{}



\maketitle

\begin{abstract}
The goal of this article is to develop BV (Batalin-Vilkovisky) formalism in the $p$-adic Dwork theory. Based on this formalism, we explicitly construct a $p$-adic dGBV algebra (differential Gerstenhaber-Batalin-Vilkovisky algebra) for a smooth projective complete intersection variety $X$ over a finite field, whose cohomology gives the $p$-adic Dwork cohomology of $X$, and its cochain endomorphism (the $p$-adic Dwork Frobenius operator) which encodes the information of the zeta function $X$. 
As a consequence, we give a modern deformation theoretic interpretation of Dwork's theory of the zeta function of $X$ and derive a formula for the $p$-adic Dwork Frobenius operator in terms of homotopy Lie morphisms and the Bell polynomials. 



\end{abstract}

\tableofcontents

\section{Introduction}

We fix a rational odd prime number $p$ and a positive integer $a$.
Let $\bF_q$ be the finite field of $q$ elements where $q=p^a$.
Let $n$ and $k$ be positive integers such that $n \geq k \geq 1$. 
We use $\underline x = [x_0, x_1, \cdots, x_n]$ as a homogeneous coordinate system of the projective $n$-space $\BP^n_{\bF_q}$\
over a finite field $\bF_q$.
Let $X=X_{\underline G}$ be a smooth complete intersection of dimension $n-k$ in the projective space $\BP^{n}_{\bF_q}$ 
and let $G_1(\ud x), \cdots, G_k(\ud x)$ be the defining homogeneous polynomials in
$\bF_q[\ud x]=\bF_q[x_0, x_1, \cdots, x_n]$ such that $\deg(G_i)=d_i$ for $i = 1, \cdots, k$.
It is well known (due to Dwork; see \cite{AdSp08} for example) that the zeta function of $X$ may be written in the form
\begin{eqnarray} \label{zpt}
Z(X/\bF_q, T) = \frac{P(T)^{(-1)^{n-k-1}}}{(1-T)(1-qT) \cdots (1-q^{n-k} T)}
\end{eqnarray}
where $P(T) \in 1 + T \bZ[T]$. The reciprocal roots of $P(T)$ are units at all non-archimedean places except places over $p$. They have absolute value $q^{(n-k)/2}$ by Deligne \cite{De74} at any archimedean prime; at $p$-adic places, the Katz conjecture saying that the $p$-adic Newton polygon lies on or above the Hodge polygon is proven by Mazur \cite{Ma73}. In \cite{AdSp08}, Adolphson and Sperber gave a new proof of Mazur's result on the Katz conjecture by computing an explicit basis for the $p$-adic Dwork cohomology.

The goal of this article is to reveal a hidden Batalin-Vilkovisky formalism in the theory of zeta functions of smooth projective complete intersection varieties defined over a finite field. BV (Batalin-Vilkovisky) formalism is a gauge-fixing method in quantum field theory, which is an important quantization method; we refer to a well-written book \cite{Mnev} on the subject. 
This leads to recapture the deformation theory of zeta functions of algebraic varieties, which was invented by Dwork (see \cite{Dw62} and \cite{Dw64} for example) and developed by Adolphson and Sperber (see \cite{AdSp06} and \cite{AdSp08} for example), using the modern deformation theoretic  view point based on the dGBV (differential Gerstenhaber-Batalin-Vilkovisky) algebra and dgla (differential graded Lie algebra). 
As an application, we derive an explicit algebraic formula for the $p$-adic Dwork Frobenius operator whose characteristic polynomial computes $Z(X/\bF_q,T)$, in terms of homotopy Lie morphisms (so called, $L_\infty$-homotopy morphisms) and the Bell polynomials.



\subsection{Main results}\label{sub1.1}

Our BV formalism for $p$-adic Dwork theory recapturing the work of Adolphson and Sperber \cite{AdSp08} leads to the following theorem.

\begin{theorem}\label{ftheorem}
There exists a $p$-adic dGBV (differential Gerstenharber-Batalin-Vilkovisky) algebra\footnote{The notation $S$ means a function defined in \eqref{dwf} associated to $G_1(\ud x), \cdots, G_k(\ud x)$.} $(\widetilde \cAb, \cdot, \tilde K_S, \ell^{\tilde K_S}, \tilde Q_S)$ over a $p$-adic field\footnote{See \eqref{kk}.} $\Bbbk$ associated to $X_{\ud G}$ and a cochain map $\Psi_S:(\widetilde \cAb, \tilde K_S) \to (\widetilde \cAb, \tilde K_S)$ which satisfy the following properties:

(a) 
 there exists an explicit $L_\infty$-endomorphism\footnote{An $L_\infty$-algebra (homotopy Lie algebra)
$(V, \ud \ell)$ is a $\Z$-graded vector space $V$ with an $L_\infty$-structure 
$\ud \ell =(\ell_1,\ell_2,\ell_3,\cdots)$, where $\ell_1$ is a differential such that $(V,\ell_1)$ is a cochain complex, $\ell_2$ is a graded Lie bracket which satisfies the graded Jacobi identity up to homotopy $\ell_3$ etc. An $L_\infty$-morphism $\ud \phi=(\phi_1,\phi_2,\cdots)$ is a morphism between $L_\infty$-algebras, say $(V,\ud \ell)$
and $(V^\prime, \ud \ell^\prime)$, such that $\phi_1$ is a cochain map of the underlying cochain complex, 
which is a Lie algebra homomorphism
up to homotopy $\phi_2$, etc. See subsection \ref{sub3.1}.}
 $\ud \phi^{\Psi_S}=(\phi^{\Psi_S}_1, \phi^{\Psi_S}_2, \cdots )$ of $(\widetilde \cAb, K_S, \ell_2^{\tilde K_S})$ such that $\phi^{\Psi_S}_1=\Psi_S$ and there is a decomposition $\widetilde \cAb =\bigoplus_{-N \leq m \leq 0} \widetilde \cAb^m$ where $\widetilde \cAb^0$ is a $p$-adic Banach commutative algebra over $\Bbbk$ and
$$
 \widetilde \cAb^{-s} = \bigoplus_{1\leq i_1 < \cdots < i_s\leq N} \widetilde \cAb^0 \cdot \eta_{i_1} \cdots \eta_{i_s}, \quad 0 \leq s \leq N.
$$

(b) there exists a separated and exhaustive decreasing filtration $\{F^i \widetilde \cAb \}_{i\in \bZ}$ consisting of $\cO_{\Bbbk}$-submodules so that $(\widetilde \cAb, \tilde K_S)$ becomes a filtered complex, and
there is a $\bF_q$-module cochain isomorphism
\begin{eqnarray*}
\cR: (F^0  \widetilde \cAb/ F^1 \widetilde \cAb, \tilde K_S) \xrightarrow{\sim} (\cA, Q_S)
\end{eqnarray*}
where $(\cA,Q_S)$ is a cdga (commutative differential graded algebra) over $\bF_q$ associated to $X$.

(c) the $0$-th cohomology $H^0(\widetilde \cAb, \tilde K_S)$ where $\tilde K_S= \ell_1^{\tilde K_S}$ is a finite dimensional $\Bbbk$-vector space whose dimension is equal to the degree of $P(T)$.

(d) the $\Bbbk$-linear completely continuous operator $\phi_1^{\Psi_S}=\Psi_S: (\widetilde \cAb, \tilde K_S) \to (\widetilde \cAb, \tilde K_S)$, which is a cochain map, satisfies 
\begin{eqnarray}\label{keyzeta}
P(q^k T)= \det(1 - T \cdot \Psi_S \big|H^0(\widetilde \cAb, \tilde K_S)).
\end{eqnarray}
\end{theorem}

The existence of the above $p$-adic dGBV algebra is motivated by 0-dimensional quantum field theory for $X$; see the subsections \ref{sub2.2} and \ref{sub4.5} (in particular, Theorem \ref{mainthm}).
The $L_\infty$-morphism in (a) of Theorem \ref{ftheorem} is obtained by applying constructions in Definition~\ref{defdesc}, to the cochain endomorphism $\Psi_S$ of $(\widetilde \cAb, \tilde K_S)$. 
We refer to the subsections \ref{sub4.4}, \ref{sub5.1}, \ref{sub5.2}, and \ref{sub5.3} for its detailed proof.

As an application, we provide a deformation formula for $\Psi_S$ using this $L_\infty$-formalism and the Bell polynomials: see Theorem \ref{stheorem} for a precise statement. 


Finding a similar construction for a relative cohomology for a family of smooth projective complete intersections (or a more general family) seems to be a nice project. We also expect that there should be an analogous construction in the case of quasi-smooth projective complete intersection X in the projective simplicial toric variety over a finite field: see our future work.

Now we briefly explain the contents of each section.
The section \ref{sec2} is devoted to explanation of a geometric idea (the Gysin sequence and the Cayley trick in the subsection \ref{sub2.1}) and a physical idea (BV formalism and quantum field theory in the subsection \ref{sub2.2}) which motivate the article, when the $p$-adic field is replaced by the field of complex numbers.

The section \ref{sec3} is about a general theory of homotopy Lie algebras and their (formal) deformation theory.
The subsection \ref{sub3.1} is devoted to a brief explanation of homotopy Lie algebras and morphisms (also called, $L_\infty$-algebras and morphisms), and dGBV algebras.
Then, in the subsection \ref{sub3.2}, we explain the descendant $L_\infty$-algebras and morphisms and their explicit relationship to the deformation theory of cochain complexes. In the subsection \ref{sub3.3}, we explain how to modify the usual deformation theory so that the deformation of the $p$-adic Dwork Frobenius operator makes sense later.

The section \ref{sec4} is the main section which develops BV formalism of the $p$-adic Dwork theory. We begin with the theory over a finite field (the subsection \ref{sub4.0}).
In the subsection \ref{sub4.1}, we spell out a key question in the $p$-adic Dwork theory.
In the subsection \ref{sub4.2}, we set up basic notations.
Then we provide a proof of main part and (a) of Theorem \ref{ftheorem} in the subsections \ref{sub4.3} and \ref{sub4.4}.
The subsection \ref{sub4.5} is devoted to physical interpretation behind $p$-adic Dwork's theory.

The final section \ref{sec5} consists of two parts; the subsections \ref{sub5.1}, \ref{sub5.2}, and \ref{sub5.3} are devoted to proofs of (b),(c), and (d) of Theorem \ref{ftheorem} respectively, and the subsection \ref{sub5.4} is about the application (Theorem \ref{stheorem}) of our theory, namely an $L_\infty$-homotopy deformation formula of the $p$-adic Dwork operator using the Bell polynomials.

%
%

\subsection{Acknowledgement}
Jeehoon Park was supported by Samsung Science \& Technology Foundation (SSTF-BA1502). 
Dohyeong Kim was supported by Research Resettlement Fund for the new faculty of Seoul National University, by Simons Foundation grant 550033, and by National Research Foundation of Korea grants 2020R1C1C1A01006819 and 2019R1A6A1A10073437.


\section{Geometric idea and quantum field theory} \label{sec2}

\subsection{The geometric idea over $\bC$} \label{sub2.1}
In order to explain the geometric idea behind the work of Dwork, Adolphson and Sperber and our idea how we came up with BV formalism of the zeta function of $X$, 
we assume in this subsection that $X$ is defined over the field of complex numbers $\bC$ instead of $\bF_q$. 

If we are interested in the cohomology of the smooth projective complete intersection variety $X$ of dimension $n-k$, then the primitive middle dimensional cohomology
$H^{n-k}_{\operatorname{prim}}(X,\bC)$ is the most interesting piece because the other degree cohomologies and non-primitive pieces can be easily described in terms of the cohomology of the projective space $\BP^n$ due to the weak Lefschetz theorem and the Poincare duality. 
For the computation of $H^{n-k}_{\operatorname{prim}}(X,\bC)$, the Gysin sequence and ``the Cayley trick'' play important roles. There is a long exact sequence, called the Gysin sequence:
\begin{eqnarray*}
\cdots \to H^{n+k-1}(\BP^n,\bC) \to H^{n+k-1}(\BP^n \setminus X, \bC) \xrightarrow{\Res_X} H^{n-k}(X,\bC) \to H^{n+k}(\BP^n,\bC)  \to \cdots
\end{eqnarray*}
where $\Res_X$ is the residue map (see p.\,96 of \cite{Dim95}). This sequence gives rise to an isomorphism 
\begin{eqnarray}\label{resx}
\Res_X:H^{n+k-1}(\BP^n \setminus X, \bC) \xrightarrow{\sim} H^{n-k}_{\operatorname{prim}}(X,\bC).
\end{eqnarray}
The Cayley trick is about translating a computation of the cohomology of the complement of a complete intersection into a computation of the cohomology of the complement of a hypersurface in a bigger space. Let $\cE=\cO_{\BP^n}(d_1) \oplus \cdots \oplus \cO_{\BP^n}(d_k)$ be the locally free sheaf
of $\cO_{\BP^n}$-modules with rank $k$. Let $\BP(\cE)$ be the projective bundle associated to $\cE$ with fiber $\BP^{k-1}$ over $\BP^n$. Then $\BP(\cE)$ is the smooth projective toric variety with Picard group isomorphic to $\bZ^2$ whose (toric) homogeneous coordinate ring is given by
\begin{eqnarray} \label{fa}
A_{\BP(\cE)}:=\bC[y_1, \cdots, y_k, x_0, \cdots, x_n]
\end{eqnarray}
where $y_1, \cdots, y_k$ are new variables corresponding to $G_1, \cdots, G_k$. Let $z_1=y_1, \cdots, z_k=y_k$ and $z_{k+1}=x_0, \cdots, z_{n+k+1}=x_{n}$ so that $A_{\BP(\cE)}=\bC[\ud z]$.
There are two additive gradings $\ch$ and $\wt$, called the charge and the weight, corresponding to the Picard group $\bZ^2$: 
$$
\ch(y_i)=-d_i, \quad \text{for } i=1, \cdots, k, \quad \ch(x_{j}) = 1, \quad \text{for }  j=0, \cdots, n,
$$ 
$$
\wt(y_i)=1, \quad \text{for } i=1, \cdots, k, \quad \wt(x_{j}) = 0, \quad \text{for }  j=0, \cdots, n. 
$$ 
Then 
\begin{eqnarray}\label{dwf}
S(\ud y, \ud x):= \sum_{j=1}^k y_j G_j(\ud x) \in A_{\BP(\cE)}
\end{eqnarray}
defines a hypersurface $X_S$ in $\BP(\cE)$.
The natural projection map $\BP(\cE) \to \BP^n$ induces a morphism 
$
\BP(\cE)\setminus X_S \to \BP^n \setminus X
$
which can be checked to be a homotopy equivalence. Hence there exists an isomorphism 
\begin{eqnarray} \label{secisom}
s^*: H^{n+k-1}(\BP(\cE) \setminus X_S, \bC) \xrightarrow{\sim} H^{n+k-1}(\BP^n \setminus X, \bC)
\end{eqnarray}
where $s$ is a section to the projection map $\BP(\cE)\setminus X_S \to \BP^n \setminus X$.
The cohomology group $H^{n+k-1}(\BP(\cE) \setminus X_S, \bC)$ of a hypersurface complement in $\BP(\cE)$ can be described explicitly in terms of the de-Rham cohomology of $\BP(\cE)$ with poles along $X_S$.
Based on this, one can further show that (see Theorem 1 in \cite{Dim95} or \cite{Ko91} and see \cite{Gr69} for the pioneering work of Griffiths in the case $k=1$, the smooth projective hypersurface case)
\begin{eqnarray}\label{gm}
\left(A_{\BP(\cE)}/ Jac(S)\right)_{\ch=c_X} \xrightarrow{\sim}   H^{n+k-1}(\BP(\cE) \setminus X_S, \bC),
\end{eqnarray}
where
$Jac(S)$ is the Jacobian ideal of $S(\ud y, \ud x)$, and 
$$
c_X := \sum_{i=1}^k d_i - (n+1).
$$
Here the subindex $\ch=c_X$ means the submodule in which the charge is $c_X$.
Note that $Jac(S)$ is the sum of the images of the endomorphisms $\prt{ S }{y_i}, \prt{S}{x_j}$ of $A_{\BP(\cE)}$ ($i =1, \cdots, k, j=0, \cdots, n$).

Note that $\BP(\cE)\setminus X_S$ is a smooth affine variety\footnote{On the other hand, $\BP^n\setminus X$ is not affine when $k > 1$.} whose coordinate ring $B=A_{\BP(\cE)}[1/S]_{\ch=0,\wt=0}$ and
$$
H^{n+k-1}(\BP(\cE) \setminus X_S, \bC)= \Omega^{n+k+1}(\BP(\cE) \setminus X_S)/d\Omega^{n+k}(\BP(\cE) \setminus X_S)=\Omega^{n+k+1}_{B}/d\Omega^{n+k}_{B}.
$$

In fact, there is an explicit map $\varphi_S:A_{\BP(\cE)} \to \Omega^{n+k+1}_{B}$ which induces an isomorphism\footnote{This is theoretically more important than \eqref{gm} for us, since this isomorphism allows us to develop BV formalism for $H^{n+k-1}(\BP(\cE) \setminus X_S, \bC)$.}
 \begin{eqnarray}\label{gmtwo}
\bar\varphi_S:A_{\BP(\cE)}/V_S    \xrightarrow{\sim}    H^{n+k-1}(\BP(\cE) \setminus X_S, \bC),
 \end{eqnarray}
where 
$V_S$ is the sum of the images of the endomorphisms $\pa{y_i}+\prt{ S }{y_i}, \pa{x_j}+\prt{S}{x_j}$ of $A_{\BP(\cE)}$ ($i =1, \cdots, k, j=0, \cdots, n$).
There is no known linear map $A_{\BP(\cE)} \to \Omega^{n+k+1}_{B}$ which induces the isomorphism in \eqref{gm}. 

These isomorphisms in \eqref{gm} and \eqref{gmtwo} lead us to 
consider the following Lie algebra representation. Let $\frg_\bC$ be an abelian Lie algebra over $\bC$ of dimension $n+k+1$. Let $u_{1}, u_2, \cdots, u_{n+k+1}$ be a ${\bC}$-basis of $\frg_\bC$. 
We associate a Lie algebra representation $\rho$ on $A_{\BP(\cE)}$ of 
$\frg_\bC$ as follows:
\begin{eqnarray*}
\rho(u_i) := \pa{z_i}+\prt{ S(\ud z) }  {z_i}, \text { for $i=1, 2, \cdots, n+k+1$}.
\end{eqnarray*}
We extend this ${\bC}$-linearly to
get a Lie algebra representation $\rho: \frg_\bC\to\End_{{\bC}}(A_{\BP(\cE)})$. Then the 0-th Lie algebra homology is isomorphic to 
$A_{\BP(\cE)}/V_S$. Also, the $(n+k+1)$-th Lie algebra cohomology is isomorphic to $A_{\BP(\cE)}/V_S$. In fact, the Chevalley-Eilenberg cohomology complex is the twisted de-Rham complex $(\Omega_{\bA^{n+k-1}}^\bullet, d + dS)$ of the affine space $\bA^{n+k-1}$. On the other hand, one can consider the Chevalley-Eilenberg homology complex.
More precisely, we use the cochain complex $(\cA_\rho^\bullet, K_\rho)$, which we call \textit{the dual Chevalley-Eilenberg complex}, such that $H^i(\cA_\rho^\bullet, K_\rho)$ $ \simeq H_{-i}(\frg_\bC, A_{\BP(\cE)})$ for $i \in \Z$:
\be
\cA_{\rho}&=&\cA_{\rho}^\bullet= A_{\BP(\cE)}[\eta_{1},\eta_2, \cdots, \eta_{n+k+1}]=\bC[\ud z][\eta_1,\eta_2 \cdots, \eta_{n+k+1}], \\
K_\rho&=&  \sum_{i=1}^N \left(\prt{ S(\ud z)}{z_i} + \pa{z_i}\right) \pa{\eta_i}:\cA_{\rho} \to \cA_{\rho}.
\ee
 We have
\bea\label{rhob}
\xymatrix{0 \ar[r] & \cA_\rho^{-(n+k+1)} \ar[r]^-{K_\rho} & \cA_\rho^{-(n+k)} \ar[r]^-{K_\rho} & \cdots \ar[r]^-{K_\rho} & \cA_\rho^{-1} \ar[r]^-{K_\rho} & \cA_\rho^{0}=A_{\BP(\cE)} \ar[r] & 0}
\eea
where
$$
 \cA_\rho^{-s} = \bigoplus_{1\leq i_1 < \cdots < i_s\leq n+k+1} A_{\BP(\cE)} \cdot \eta_{i_1} \cdots \eta_{i_s}, \quad 0 \leq s \leq n+k+1.
$$
Since one easily sees that $H^{n+k-1-s}(\Omega_{\bA^{n+k-1}}^\bullet, d + dS) \xrightarrow{\sim}
H^s(\cA_\rho^\bullet, K_\rho)$ for $s \in \bZ$ (using the fact that $\frg_\bC$ is abelian), i.e. their differential module structures are isomorphic, either complexes can be used to study the primitive middle dimensional cohomology of $X$. 
In general, people preferred to use the twisted de-Rham complex (for example, \cite{AdSp08} and \cite{Dim95}), sometimes called the algebraic Dwork complex over $\bC$.

But our key observation is that $(\cA_\rho^\bullet, \cdot, K_\rho)$ provides a model for  BV formalism of certain 0-dimensional quantum field theory; its quantum master equation (the Maurer-Cartan equation of certain dgla) governs a deformation theory of a cochain complex $(\cA_\rho^\bullet, \cdot, K_\rho)$ and its cochain maps.
In the main body of the paper, we will develop BV formalism over $\bF_q$ and a $p$-adic field instead of $\bC$, and provide a modern deformation theoretic interpretation of Dwork's theory of the zeta function of $X$ over $\bF_q$. 


\subsection{Physical model: BV formalism and 0-dimensional quantum field theory}\label{sub2.2}

We continue to work with archimedean fields, namely, real numbers $\bR$ or complex numbers $\bC$. 

Here we set up a 0-dimensional field theory $\QFT_S^\bC$ and explain its BV formalism based on the subsection \ref{sub2.1}. BV formalism is a way of understanding the Feynman path integral in quantum field theory. A classical BV formalism consists of $\bZ$-graded odd-symplectic manifold $\cM$ with action functional and a quantum BV formalism consists of a Berezinian measure compatible with the odd-symplectic structure and the quantum master equation with BV Laplacian. This leads to the notion of dGBV (differential Gerstenhaber-Batalin-Vilkovisky) algebra. A gauge-fixing in BV formalism is given by a choice of a Lagrangian submanifold of $\cM$. We refer to \cite[Chapter 4]{Mnev} for relevant basic notions and physical and mathematical significance of classical and quantum BV formalism.

\begin{defn} \cite[Section 4.8.1]{Mnev} \label{bv1}
Classical BV formalism $\operatorname{CBV}_S$ consists of the following data:
\begin{itemize}
\item a $\bZ$-graded supermanifold $\cF$ (the space of BV fields); \cite[Section 4.2]{Mnev}.
\item an odd symplectic structure, i.e. a differential 2-form $\o \in \O^2(\cF)^{-1}$ with associated Poisson bracket  $\{\cdot, \cdot\}$; \cite[Section 4.4]{Mnev}.
\item a BV action functional $S \in C^\infty(\cF)^0$ such that $\{S, S\}=0$ (where $C^\infty(\cF)$ means the smooth functions on $\cF$; \cite[Definition 4.2.1]{Mnev}.
\item a vector field $Q \in \cX(\cF)^1$ defined by $Q(\cdot)=\{S, \cdot\}$, which satisfies that $\iota_Q \omega = dS$; \cite[Definition 4.2.10]{Mnev}.
\end{itemize}
\end{defn}

\begin{defn} \cite[Section 4.8.2]{Mnev} \label{bv2}
Quantum BV formalism $\operatorname{QBV}_S$ for 0-dimensional field theory consists of the following data:
\begin{itemize}
\item a $\bZ$-graded supermanifold $\cF$ (the space of BV fields).
\item an odd symplectic structure, i.e. a differential 2-form $\o \in \O^2(\cF)^{-1}$ with associated Poisson bracket  $\{\cdot, \cdot\}$; \cite[Section 4.4.2]{Mnev}.
\item a Berezinian measure $\mu_0$ compatible with $\o$; \cite[Section 3.8]{Mnev}, \cite[Definition 4.4.10]{Mnev}.
\item An extended BV action functional $S=S^{(0)} -i \hbar S^{(1)} +(-i\hbar)^2 S^{(2)} +\cdots $ satisfying a quantum master equation
$$
\frac{1}{2}\{S,S\} -i \hbar \Delta_{\mu_0} S=0
$$
where $\Delta_{\mu}$ is the BV Laplacian associated to $\mu$; \cite[Section 4.4.3]{Mnev}

\end{itemize}
\end{defn}

\begin{defn}\label{bv3}
If $S^{(0)}$ in quantum BV formalism $\operatorname{QBV}_S$ satisfies $\{S^{(0)}, S^{(0)}\}=0$, we say that $\operatorname{QBV}_S$ is a BV quantization of classical BV formalism $\operatorname{CBV}_{S^{(0)}}$. It means that quantum BV formalism modulo $\hbar$ is classical BV formalism (taking a classical limit corresponds to the Planck constant $\hbar \to 0$).
\end{defn}

The space-time for $\QFT_S^\bC$ is (0+0)-dimensional, i.e. a finite set of points of cardinality $N=n+k+1$.
The space of fields is the space of functions on this space-time, i.e. the affine space $\bA^N$ of dimension $N$ with coordinates $\ud z$. We define the action functional for $\QFT_S^\bC$ as the function $S(\ud z)$ on $\bA^N$ in \eqref{dwf}.
The equation of motion space is given by solutions of the Euler-Lagrange equation for $\QFT_S^\bC$, i.e. the critical locus of $S$, $d (S (\ud z))=0$:
\begin{eqnarray}\label{criticallocus}
G_1(\ud x) = \cdots = G_k(\ud x) =0, \quad \frac{\partial}{\partial x_i} S(\ud z) =0, \quad i =0, 1, \cdots, n.
\end{eqnarray}
Notice that the weight zero part of the equation of motion space is given by $G_1(\ud x)=\cdots=G_k(\ud x)=0$, which are defining equations of $X=X_{\ud G}$ in $\BP^n$. We define the classical observables as the polynomial functions on the equation of motion space, i.e. elements of $A_{\BP(\cE)}/ Jac(S)$. The $\ch$ and $\wt$ can be understood as a gauge action of the abelian group $\bC^\times \times \bC^\times$ on the space of fields $\bA^N$; the action functional $S(\ud z)$ is homogeneous under the gauge action. If we introduce $\hbar$ such that $\ch(\hbar)=0$ and $\wt(\hbar)=1$, then $\frac{S(\ud z)}{\hbar}$ is invariant under the gauge action.

Let $\cF_\bC$ be a supermanifold whose algebraic structure sheaf is given by $\cA_{\rho}$ in \eqref{rhob}; $\cF_\bC$ is a (-1)-shifted cotangent bundle\footnote{The fiber coordinate has grading -1.} $T^*[-1](\bC^N)$.
We fix a Darboux coordinate $\ud \eta=(\eta_1, \cdots, \eta_N), \ud z=(z_1, \cdots, z_N)$ of $\cF$.
Then $\o=\sum_{i=1}^N d\eta_i \wedge d z_i$ defines an odd symplectic structure on $\cF$.
The associated Poisson structure is given by 
\be
\{ f, g\}= \sum_{i=1}^N \frac{\partial f}{\partial z_i} \frac{\partial g}{\partial \eta_i} - (-1)^{|f|+1} \frac{\partial f}{\partial \eta_i} \frac{\partial g}{\partial z_i}, \quad f, g \in \widetilde \cAb.
\ee
By definition, the Berezinian measure compatible with $\o$ is given by $\mu_0:= d^N \ud x D^N \ud \eta$ using the Darboux coordinate $\ud \eta=(\eta_1, \cdots, \eta_N), \ud z=(z_1, \cdots, z_n)$ of $\cF$. Then the BV Laplacian associated to $\mu_0$ is given by
\be
\Delta_{\mu_0}:=\sum_{i=1}^N\frac{\partial}{\partial z_i}\frac{\partial}{\partial\eta_i},
\ee
Moreover, the BV Laplacian associated to $\mu=e^{2S}\mu_0=e^{2 S} d^N \ud x D^N \ud \eta$ is given by (\cite[(4.4.6)]{Mnev})
\be
\Delta_{\mu}= \Delta_{\mu_0} + \{S, \cdot\}.
\ee

In \cite{KPP21}, it is shown that the period integral for $H^{n-k}_{\operatorname{prim}}(X,\bC)$ can be understood as (a form of) Feynman path integral of $\QFT_S^\bC$ by a careful analysis of the isomorphism $\Res_X \circ s^* \circ  \varphi_S$, where relevant definitions are given in \eqref{resx}, \eqref{secisom}, and  \eqref{gmtwo}. Roughly speaking,  the period integral $\int_\g \o$ for $[\o] \in H^{n-k}_{\operatorname{prim}}(X,\bC)$ and a fixed vanishing homology cycle $[\g]\in H_{n-k}(X,\bZ)$, can be understood as
\bea\label{fpi}
\int_\g \o = \int_{\bA^N} f_\o(\ud z) e^{S(\ud z)} d\mu_\g
\eea
for ``some measure'' $\mu_\g$ and some function\footnote{$[\o]=(\Res_X \circ s^* \circ \varphi_S)(f_\o)$.} $f_\o$ on $\bA^N$: we refer to \cite[the proof of Theorem 1.1]{KPP21} for its precise explanation. 
In other words, we can understand the period integral $\o \mapsto \int_{\g}\o$ as a BV integral (in the sense of \cite[Section 4.4.4]{Mnev}), a $\bC$-linear functional $f \mapsto \cC_\g^X(f)$ on $\cA_{\BP(\cE)}$ such that $\cC_\g \circ \Delta_{\mu}=0$;
the authors showed that the kernel of the map $f \mapsto \int_{\bA^N} f e^{S(\ud z)} d\mu_\g$
is $V_S=\Delta_{\mu}(\cA_{\rho}^{-1})$ in \cite[Proposition 3.7]{KPP21}.
This leads us to define the quantum observables as elements of $A_{\BP(\cE)}/ V_S=A_{\BP(\cE)}/ \Delta_{\mu}(\cA_{\rho}^{-1})$.
If we consider the integral $\int_{\bA^N} f_\o(\ud z) e^{i \frac{S(\ud z)}{\hbar}} d\mu_\g$,
then $S \in \cA_\rho^0=A_{\BP(\cE)}$ satisfies the quantum master equation 
$$
\frac{1}{2}\{S,S\} -i \hbar \Delta_{\mu_0} S=0;
$$
$T^*[-1](\bC^N)$ with $S=S^{(0)}, \o,$ and $\mu_0$ gives us quantum BV formalism in Definition \ref{bv2}.

The critical locus $dS=0$ in $\bA^N$ is not discrete as we saw in \eqref{criticallocus}.
On the other hand, if we restrict our space of fields to $U_N:=\bC^N-(\bC^k \times \{ 0\} \sqcup \{0\} \times \bC^{n+1})=(\bC^k-\{0\}) \times (\bC^{n+1}-\{0\})$, then the critical locus $dS=0$ becomes discrete\footnote{This removes the necessity of BV gauge-fixing to understand the period integrals (the Feynman path integrals  in $\QFT_S$; a main point for gauge-fixing is to modify the action functional so that it has isolated critical points. See \cite[p78, p97, p117]{Mnev} for details.} in $U_N$ due to the smoothness of $X$. Then $\bC^\times \times \bC^\times$ acts (``gauge action'') on $U_N$
\begin{eqnarray*}
(c,w) \cdot (y_1, \cdots, y_k, x_0, \cdots, x_n)= (w c^{-d_1}y_1, \cdots, w c^{-d_k}y_k, c x_0, \cdots, c x_n)
\end{eqnarray*}
for $(c,w)\in \bC^\times \times \bC^\times$. The projective bundle $\BP(\cE)$ appeared in \eqref{fa} can be realized as a geometric quotient $U_N/\bC^\times \times \bC^\times$ and $X_S$ defines a smooth hypersurface in $\BP(\cE)$.

\section{Formal deformation theory of cochain complexes with multiplication} \label{sec3}

\subsection{Homotopy Lie algebra and BV algebra} \label{sub3.1}

Roughly speaking, a homotopy Lie algebra ($L_\infty$-algebra) is  a ``differential Lie algebra up to homotopy''. We refer to section 13.2, \cite{LV12} for the precise theoretical definition (as an algebra over a particular algebraic operad $\operatorname{Lie}$) and its basic properties.
Here we only briefly review an explicit description following the appendix 5.2, \cite{PP}.

Let $k$ be a field of characteristic zero (we need this since we have to divide $m!$ in the definition of homotopy Lie algebras). 
 Let $\art_{k}$ denote the category of $\bZ$-graded artinian local $k$-algebras with residue field $k$ 
 and $\widehat{\art_k}$ be the category of complete $\bZ$-graded noetherian local $k$-algebras. 
For $\ma \in \hbox{Ob}(\art_{k})$, $\mm_\ma$ denotes the maximal ideal of $\ma$ which is a nilpotent $\Z$-graded super-commutative and associative $k$-algebra without unit. 
Let $V= \bigoplus_{i\in \bZ} V^i$ be a $\bZ$-graded vector space over $k$. If $x \in V^i$, we say that $x$ is a homogeneous element of degree $i$; let $|x|$ be the degree of a homogeneous element of $V$. For each $n \geq 1$ let $S(V)=\bigoplus_{n=0}^\infty S^n(V)$
be the free $\bZ$-graded super-commutative and associative algebra over $k$ generated by $V$, which is the quotient algebra of the free tensor algebra $T(V)=\bigoplus_{n=0}^\infty T^n(V)$ by the ideal generated by $x\otimes y - (-1)^{|x||y|}y\otimes x$. Here $T^0(V)=k$ and $T^n(V) = V^{\otimes n}$ for $n \geq 1$.

\begin{definition}[$L_{\infty}$-algebra]
\label{shl}
The triple $V_L=(V,\underline{\ell}, 1_{V})$ is a unital $L_\infty$-algebra over $k$  if
$1_V \in V^0$ and $\underline{\ell}=\ell_1,\ell_2,\cdots$ be a family of
$k$-linear maps such that

\begin{itemize}
\item $\ell_n \in \Hom(S^n (V),V)^1$ for all $n\geq 1$.
\item $\ell_n(v_1,\cdots, v_{n-1}, 1_{V})=0$,  for all $v_1,\cdots, v_{n-1} \in V$, $n\geq 1$. 
\item for any $\ma \in \hbox{Ob}\left(\art_{k}\right)$ and for all $n\geq 1$
\be
\sum_{k=1}^n\frac{1}{(n-k)! k!} {\ell}_{n-k+1}\left({\ell}_k\left( \g,\cdots,  \g\right), \g,\cdots,
 \g\right)=0,
\ee
whenever $\g \in (\mm_\ma\otimes V)^0$, where 
\be
&&{\ell}_n\big(a_1\otimes v_1, \cdots,  a_n\otimes v_n\big) \\
&&=(-1)^{|a_1|+|a_2|(1+|v_1|) +\cdots + |a_n|(1+|v_1|+\cdots +|v_{n-1}|)}
 a_1\cdots a_n\otimes \ell_n\left(v_1,\cdots, v_n\right).
\ee
\end{itemize}
\end{definition}


\begin{definition} [$L_{\infty}$-morphism]
\label{shlm}
A morphism of unital $L_\infty$-algebras from $V_L$ into $V'_L$ 
is a  family $\underline{\phi}=\phi_1,\phi_2,\cdots$ 
such that
\begin{itemize}
\item $\phi_n \in \Hom (S^n V, V')^0$ for all $n\geq 1$.
\item $\phi_1(1_{V})=1_{V'}$ and $\phi_n(v_1,\cdots, v_{n-1}, 1_{V})=0$, $v_1,\cdots, v_{n-1} \in V$, for all $n\geq 2$. 
\item for any $\ma \in \hbox{Ob}\left(\art_k\right)$ and for all $n\geq 1$
\be
&&\sum_{j_1+ j_2 =n}
\frac{1}{j_1! j_2!} {\phi}_{j_1+1}\left( \ell_{j_2}(\g,\cdots,\g), \g,\cdots, \g\right) \\
&&=
\sum_{j_1+\cdots + j_r =n}\frac{1}{r!}\frac{1}{j_1!\cdots j_r!}
 \ell'_r\left({\phi}_{j_1}(\g,\cdots,\g), \cdots, {\phi}_{j_r}(\g,\cdots,\g)\right),
\ee
whenever $\g \in (\mm_\ma\otimes V)^0$,
where
\be
&&{\phi}_n\big(a_1\otimes v_1, \cdots, a_n\otimes v_n\big)\\
&&=(-1)^{|a_2||v_1| +\cdots + |a_n|(|v_1|+\cdots +|v_{n-1}|)}
 a_1\cdots a_n\otimes \phi_n\left(v_1,\cdots,v_n\right).
\ee
\end{itemize}
\end{definition}
If we forget the unity $1_V$, then we call a pair $(V, \ud \ell)$ an $L_\infty$-algebra. We can similarly define an $L_\infty$-morphism without the condition on $1_V$.
One can define the composition of $L_\infty$-morphism and it can be checked that unital $L_{\infty}$-algebras over $k$ and $L_\infty$-morphisms form a category.

\bed
The cohomology $H$ of the $L_{\infty}$-algebra $(V, \underline \ell)$ is the cohomology of the underlying complex $(V, K=\ell_1)$. An $L_{\infty}$-morphism $\underline \phi$ is an $L_\infty$-quasi-isomorphism if $\phi_1$ induces an isomorphism on cohomology.
\eed

\bed
An $L_\infty$-algebra $(V, \ud \ell)$ is called a (shifted) dgla(differential graded Lie algebra)
if $\ell_m =0$ for $m\geq 3$.
This means that $(V, \ell_1)$ is a cochain complex ($\ell_1$ has degree 1), i.e. $\ell_1\circ \ell_1=0$ and $(V,\ell_2)$ is a graded Lie bracket ($\ell_2$ has also degree 1), i.e.
\be
\ell_2(x_1, x_2)&=&(-1)^{|x_1||x_2|}\ell_2(x_2, x_1) \\
0&=&\ell_2(\ell_2(x_1,x_2), x_3) + (-1)^{|x_1|} \ell_2(x_1, \ell_2(x_2, x_3))
+(-1)^{(|x_1|+1)|x_2|}\ell_2(x_2, \ell_2(x_1,x_3)),
\ee
and $\ell_1$ is a graded derivation of the Lie bracket
\be
\ell_1\left(\ell_2(x_1,x_2)\right) =- \ell_2(\ell_1(x_1), x_2) + (-1)^{|x_1|+1}\ell_2(x_1, \ell_1(x_2)).
\ee
\eed

In fact, any $L_\infty$-algebra can be strictified to give a dgla, i.e. any $L_\infty$-algebra is $L_\infty$-quasi-isomorphic to a dgla.
Now we give definitions of G-algebra, GBV algebra, and dGBV algebra, slight variations of a dgla, which are suitable for our analysis on the $p$-adic Dwork complex.

\begin{defn}\label{bvd}
Let $k$ be a field. Let $(\cC,\cdot)$ be a unital $\Z$-graded super-commutative and associative $k$-algebra.
Let $[\bullet, \bullet]: \cC \otimes \cC \to \cC$ be a bilinear map of degree 1.

(a) $(\cC, \cdot, [\cdot, \cdot])$ is called a G-algebra (Gerstenhaber algebra) over $k$ if
\begin{align*}
  \quad [a, b] &= (-1)^{|a||b|}[b,a],\\
  [a, [b,c]]&= (-1)^{|a|+1}[[a,b],c]+(-1)^{(|a|+1)(|b|+1)} [b, [a,c]],\\
 \quad [a,b \cdot c]&= [a, b] \cdot c +(-1)^{(|a|+1)\cdot |b|} b \cdot [a,c],
 \end{align*}
 for any homogeneous elements $a, b, c \in \cC.$

(b) $(\cC,\cdot, K,\ell_2^K)$ is called a GBV(Gerstenhaber-Batalin-Vilkovisky)-algebra\footnote{$(\cC,\cdot, K)$ is called a BV algebra, if $(\cC,\cdot, K,\ell_2^K)$ is a GBV algebra} over $k$ where
\begin{align}\label{elltwo}
\ell_2^{K}(a,b):= K(a \cdot b)-K(a)\cdot b -(-1)^{|a|} a\cdot K(b), \quad a,b \in \cC,
\end{align}
 if 
\text{$(\cC, K, \ell_2^K)$ is a (shifted) dgla and} 
\text{$(\cC, \cdot, \ell_2^K)$ is a G-algebra,}

%

 (c) $(\cC, \cdot, K, \ell_2^K, Q)$, where $Q:\cC \to \cC$ is a linear map of degree 1, is called a dGBV(differential Gerstenhaber-Batalin-Vilkovisky) algebra if $(\cC, \cdot, K,  \ell_2^K(\cdot, \cdot))$ is a GBV algebra
 and $(\cC,\cdot,Q)$ is a cdga(commutative differential graded algebra) with $KQ+QK=0$, i.e.
 $$
 Q^2=0, \quad Q(a \cdot b) = Q(a) \cdot b + (-1)^{|a|} a \cdot Q(b), \quad a,b \in \cC.
 $$
\end{defn}

\subsection{Formal deformation theory} \label{sub3.2}

In this subsection, we will study the deformations of the data $(\cC,\cdot,K,\Psi)$ where
\begin{quote}
(1) $(\cC,\cdot)$ is a $\mathbb{Z}$-graded super-commutative associative algebra algebra over $k$,\\
(2) $(\cC,K)$ is a cochain complex over $k$, and\\
(3) $\Psi:\cC\rightarrow\cC$ is a $k$-linear cochain map, i.e. $\Psi \circ K= K\circ \Psi$.
\end{quote}
Its deformation theory was studied in \cite[Subsection 3.4]{PP}. We briefly review it here.

\begin{definition}\label{partition}
A \emph{partition} $\pi= B_1 \cup B_2\cup \cdots$ of the set $[n]=\{1,2, \cdots, n\}$ is a decomposition of $[n]$ into a pairwise disjoint non-empty subsets $B_i$, called \emph{blocks}. Blocks are ordered by the minimum element of each block and each block is ordered by the ordering induced from the ordering of natural numbers. The notation $|\pi|$ means the number of blocks in a partition $\pi$ and $|B|$ means the size of the block $B$. If $k$ and $k'$ belong to the same block in $\pi$, then we use
the notation $k \sim_\pi k'$. Otherwise, we use $k \nsim_\pi k'$. Let $P(n)$ be the set of all partitions of $[n]$. 
\end{definition}

%

\bed \label{defdesc}
For a given $(\cC, \cdot, K)$, we define $(\cC, \underline \ell^K)$, where $\underline \ell^K=\ell_1^K, \ell_2^K, \cdots$ is the family of linear maps $\ell_n^K: S^n(\cC) \to \cC$, inductively defined by the formula: $\ell_1^{K}=K$ and
\be
&&\ell_n^{K}(x_1, \cdots, x_{n-1}, x_n)=  \ell_{n-1}^{K}(x_1,\cdots, x_{n-2}, x_{n-1}\cdot x_n)\\
&&
-\ell_{n-1}^{K}(x_1, \cdots, x_{n-1}) \cdot x_{n} 
-(-1)^{|x_{n-1}|(1+|x_1|+\cdots + |x_{n-2}|)}  x_{n-1}\cdot \ell_{n-1}^{K}(x_1, \cdots, x_{n-2}, x_n), \quad n \geq 2,
\ee
for any homogeneous elements $x_1, x_2, \cdots, x_n \in \cC$.

 For a given cochain map $f: (\cC, \cdot, K) \to (\cC', \cdot, K')$, we define $ \underline \phi^f=\phi_1^f, \phi_2^f, \cdots $ as a family of $k$-linear maps $\phi_n^f: S^n(\cC) \to \cC'$ defined inductively by the formula: $\phi_1^{f}= {f}$ and
 \be
\phi_m^{f}(x_1, \cdots, x_m) = \phi_{m-1}^{f} (x_1, \cdots, x_{m-2}, x_{m-1}\cdot x_m) -
\sum_{\substack{\pi \in P(m), |\pi|=2 \\ m-1 \nsim_\pi m}} \phi^{f} (x_{B_1})\cdot \phi^{f}(x_{B_2}), \quad m \geq 2,
\ee
 for any homogeneous elements $x_1, x_2, \cdots, x_m \in \cC$. 
  Here we use the following notation:
\be
x_B&=& x_{j_1} \otimes \cdots \otimes x_{j_{r}} \text{ if }  B=\{j_1, \cdots, j_{r}\},\\
\phi^f(x_{B})&=&\phi^f_r(x_{j_1}, \cdots, x_{j_r}) \text{ if } B=\{j_1, \cdots, j_r\}.
\ee
\eed

The following was proved in \cite[Subsection 3.2]{PP}.
\begin{proposition}
$(\cC, \ud \ell^K)$ is an $L_\infty$-algebra and $\ud \phi^f$ is an $L_\infty$-morphism from $(\cC, \ud \ell^K)$ to $(\cC', \ud \ell^{K'})$.
\end{proposition}
We call the above $L_\infty$-algebras and $L_\infty$-morphisms as descendant $L_\infty$-algebras and descendant $L_\infty$-morphisms respectively.

\begin{rem}
The construction of $(\cC, \ud \ell^K)$ can be viewed as a generalization of the BV bracket \eqref{elltwo} for a differential operator $L$ of order $\leq 2$ to a differential operator $K$ of any finite order. Note that the $L_\infty$-algebra $(\cC, K, 0, 0, \cdots )$ is isomorphic to $(\cC, \ud \ell^K)=(\cC, \ell_1^K=K, \ell_2^K, \ell_3^K, \cdots ) $ as $L_\infty$-algebras.

\vspace{1em}
For $\mathfrak{a}\in\art_k$ denote $\mathfrak{m_a}$ its maximal ideal. In what follows we endow $\mathfrak{a}\otimes\cC$ the natural $\mathbb{Z}$-grading.

According to \cite[Lemma 3.1]{PP}, if $\Gamma\in(\mathfrak{m_a}\otimes\cC)^0$ satisfies the Maurer-Cartan equation:
\begin{align*}
K(e^\Gamma-1)=0\iff \sum_{n\geq1}\frac{1}{n!}\ell^K_n(\Gamma,\cdots,\Gamma)=0
\end{align*}
then 
\begin{align}\label{kg}
K_\G:=e^{-\Gamma}\circ K\circ e^\Gamma
\end{align}
becomes a differential on $\mathfrak{a}\otimes\cC$, i.e. $(\mathfrak{a}\otimes\cC, K_\Gamma)$ is again a cochain map, which is a (formal) deformation of $(\cC,K)$ by the Maurer-Cartan solution $\Gamma$. 

Now we deform a cochain map $f:(\cC,\cdot, K)\to (\cC',\cdot, K')$. If we assume that $K(e^\Gamma-1)=0$ for some $\Gamma\in(\mathfrak{m_a}\otimes\cC)^0$, and $K'(e^{\Gamma'}-1)=0$ for some $\Gamma'\in(\mathfrak{m_a}\otimes\cC')^0$, then
\begin{align*}
\xymatrix{f_{\Gamma',\Gamma}:=e^{-\Gamma'}\circ f \circ e^\Gamma:\mathfrak{a}\otimes\cC \ar[r] & \mathfrak{a}\otimes\cC'},
\end{align*}
is clearly a cochain map from $(\mathfrak{a}\otimes \cC, K_\Gamma)$ to $(\mathfrak{a}\otimes \cC', K'_{\Gamma'})$.
In particular, we can deform a cochain endomorphism $\Psi:(\cC,K)\rightarrow(\cC,K)$ using a Maurer-Cartan solution $\Gamma$ and $\Psi_\G:=\Psi_{\G,\G}$. 

Unfortunately, this formal deformation is not suitable for a $p$-adic deformation of a cochain endomorphism of the $p$-adic Banach algebra $(\widetilde{A(b)}, \cdot, \tilde \Delta)$ in (\ref{pb}) and (\ref{pbd}).\footnote{We will explain this further later in subsection \ref{sub3.3}.}
Thus, in the next subsection, we study a slightly enhanced deformation theory.

\end{rem}

\subsection{Deformation theory for the $q$-power map}\label{sub3.3}

For a $p$-adic deformation of the $p$-adic Banach algebra $(\widetilde{A(b)}, \cdot, \tilde \Delta)$, we consider a graded $k$-algebra endomorphism\footnote{In our application to the zeta function of $X$, $\mathfrak{a}$ will be a formal power series ring with $k$ number of variables $t_1, \cdots, t_k$ and $\sigma$ will be a $q$-power map of those variables.} $\sigma:\mathfrak{a}\rightarrow\mathfrak{a}$. Since $\sigma$ is an algebra map, we have $\sigma(\mathfrak{m_a})\subseteq\mathfrak{m_a}$ by the nilpotency of $\mathfrak{m_a}$, and
\begin{align*}
Ke^{\sigma\Gamma}=K(\sigma e^\Gamma)=\sigma(Ke^\Gamma)
\end{align*}
so the triple $(\mathfrak{a}\otimes\cC,\cdot,K_{\sigma\Gamma})$, where $K_{\sigma\Gamma}=e^{-\sigma\Gamma} \circ K\circ e^{\sigma\Gamma}$, is also a cochain complex whenever $Ke^\Gamma=0$. 
Let $\Psi$ be a cochain endomorphism of $(\cC, K)$. Now define
\begin{align} \label{psg}
\xymatrix{\Psi_{\sigma,\Gamma}:=e^{-\sigma\Gamma}\circ\Psi\circ e^\Gamma:\mathfrak{a}\otimes\cC \ar[r] & \mathfrak{a}\otimes\cC}
\end{align}
Then
\begin{align*}
\Psi_{\sigma,\Gamma}\circ K_\Gamma&=e^{-\sigma\Gamma}\circ\Psi\circ e^\Gamma\circ e^{-\Gamma}\circ K\circ e^\Gamma\\
&=e^{-\sigma\Gamma}\circ\Psi\circ K\circ e^\Gamma\\
&=e^{-\sigma\Gamma}\circ K\circ\Psi\circ e^\Gamma\\
&=e^{-\sigma\Gamma}\circ K\circ e^{\sigma\Gamma}\circ e^{-\sigma\Gamma}\circ\Psi\circ e^\Gamma\\
&=K_{\sigma\Gamma}\circ\Psi_{\sigma,\Gamma}
\end{align*}
shows that
\begin{align*}
\xymatrix{\Psi_{\sigma,\Gamma}:(\mathfrak{a}\otimes\cC,\cdot,K_\Gamma) \ar[r] & (\mathfrak{a}\otimes\cC,\cdot,K_{\sigma\Gamma})}
\end{align*}
defines a cochain map. Since $K_\Gamma\neq K_{\sigma\Gamma}$ in general, $\Psi_{\sigma,\Gamma}$ is not an endomorphism in general. However, multiplication by $e^x$ has inverse $e^{-x}$ so via the commutative diagram
\begin{align*}
\xymatrix{
H^\bullet(\mathfrak{a}\otimes\cC,K_\Gamma) \ar[d]^-\wr_-{e^\Gamma} \ar[r]^-{\Psi_{\sigma,\Gamma}} & H^\bullet(\mathfrak{a}\otimes\cC,K_{\sigma\Gamma}) \ar[d]^-\wr_-{e^{\sigma\Gamma}} \\
\mathfrak{a}\otimes H^\bullet(\cC,K) \ar[r]_-\Psi & \mathfrak{a}\otimes H^\bullet(\cC,K)
}
\end{align*}
We may regard $\Psi_{\sigma,\Gamma}$ as an endomorphism on the cohomology space. Moreover, this diagram also shows that $\Psi_{\sigma,\Gamma}$ on the cohomology space depends only on $\Psi$ on the cohomology space.

\section{$p$-adic quantum BV formalism for $X$}\label{sec4}

Now we develop the main construction over $p$-adic fields.
We start from classical BV formalism over a finite field $\bF_q$.

\subsection{Classical BV formalism over $\bF_q$}\label{sub4.0}
Now we come back to a smooth projective complete intersection variety $X$ defined over $\bF_q$. The subsection \ref{sub2.1} suggests to consider the dual Chevalley-Eilenberg complex defined over $\bF_q$.
Let $N=n+k+1$ and 
$$
A := {\bF_q}[z_\m]_{\m=1,\cdots, N}=\bF_q [\ud z]
$$ 
where we recall $z_1=y_1, \cdots, z_k=y_k$ and $z_{k+1}=x_0, \cdots, z_N=x_{n}$.
We consider the Dwork potential
$$
S(\ud z) := \sum_{\ell=1}^k y_\ell \cdot G_{\ell}(\underline{x})
 \in \bF_q[\ud z].
$$

%
The same procedure as the subsection \ref{sub2.1} provides us a $\Z$-graded super-commutative algebra $\cA$ with differential $Q_{S}$:

\begin{align}\label{finite}
\begin{split}
\cA&=\cA^\bullet= {\bF_q}[\ud z][\ud \eta]={\bF_q}[z_1, z_2, \cdots, z_N][\eta_{1},\eta_2, \cdots, \eta_N], \\
Q_S&:=\sum_{i=1}^N \prt{S(\ud z)}{z_i} \pa{\eta_i}:\cA \to \cA.
\end{split}
\end{align}
Then we have
$
H^0(\cA, Q_S) \xrightarrow{\sim} A/Jac(S)
$
where $Jac(S)$ is the Jacobian ideal of $S$. We obtain the following proposition.

\begin{prop}
$(\cA, \cdot, Q_S)$ is a cdga (commutative differential graded algebra) over $\bF_q$, which provides classical BV formalism with action functional $S$, odd symplectic form $\o=\sum_{i=1}^N d\eta_i \wedge d z_i$, and the associated Poisson bracket is given by 
\be
\{ f, g\}:= \sum_{i=1}^N \frac{\partial f}{\partial z_i} \frac{\partial g}{\partial \eta_i} - (-1)^{|f|+1} \frac{\partial f}{\partial \eta_i} \frac{\partial g}{\partial z_i}, \quad f, g \in \cA.
\ee
\end{prop}
\begin{proof}
The definition \ref{bv1} makes sense over $\bF_q$, since the definition is algebraic. The $\bZ$-graded algebra $\cA$ is a global section of an algebraic structure sheaf for the space $\cM_q=T^*[-1] \bF_q^N$ of BV fields, which is the (-1)-shifted cotangent bundle of the affine space $\bF_q^N$.
\end{proof}

\subsection{Key question for $p$-adic theory} \label{sub4.1}
The main question is whether one can find a 
$p$-adic lift of $(\cA, \cdot, Q_S)$, which has a $p$-adic complete continuous endomorphism $\Psi_S$ as cochain map such that the characteristic polynomial of $\Psi_S$ is equal to the zeta function $P(q^k T)$ in (\ref{zpt}). More precisely, can we find a $\bZ$-graded $p$-adic Banach algebra $\tilde \cA$ with the differential $\tilde K_S$ and a filtration $\{F^i \tilde \cA \}_{i\in \bZ}$ such that
there is a $\bF_q$-module isomorphism
\begin{eqnarray*}
\cR: (F^0 \tilde \cA/ F^1 \tilde \cA, \tilde K_S) \xrightarrow{\sim} (\cA, Q_S)
\end{eqnarray*}
and there is a cochain endomorphism $\Psi_S$ of $(\tilde \cA, \tilde K_S)$
which is $p$-adic completely continuous and whose characteristic polynomial is equal to $P(q^k T)$?
This question was essentially answered by Dwork and his successors (notably, Adolphson and Sperber). There are two technical difficulties for achieving this.
As is well-known, the differential operators like $\pa{z_i}$ behave badly in characteristic $p$. Even in the $p$-adic case, the differential operators behaves differently from the complex analytic case: the Poincare lemma fails if one considers $p$-adic analytic functions on the closed unit disc.
Another difficulty is the $p$-adic convergence problem of the exponential function. In the complex analytic case, the radius of convergence of the exponential function is infinity, but in the $p$-adic case, the radius of convergence is $p^{-\frac{1}{p-1}} < 1$. All of these difficulties were resolved by Dwork by introducing an overconvergent module and the splitting function.
Here we follow the version of Adolphson and Sperber, \cite{AdSp08}.

Our academic contribution is to change the product structure on the $p$-adic twisted de-Rham complex in \cite{AdSp08} (by using the $p$-adic dual Chevalley-Eilenberg complex) in order to reveal the BV structure, which put us in a natural framework of 0-dimensional quantum field theory and modern deformation theory.\footnote{There is a motto that every deformation problem in characteristic zero can be controlled by the Maurer-Cartan equation of some homotopy Lie algebra.}

\subsection{Basic notions}\label{sub4.2}
Let $\zeta_p$ be a primitive $p$-th root of unity in $\bC_p$, 
where $\bC_p$ is the $p$-adic completion of the algebraic closure of $\bQ_p$. 
We fix a $p$-adic absolute value $|\cdot|_p$ and a $p$-adic valuation $\val_p$ on $\bC_p$ such that $\val_p(p)=1$ and $|x|_p = p^{-\val_p(x)}$.

For the $p$-adic overconvergent module, we fix a rational number $b \in \bQ$ such that 
\begin{align}\label{bcond}
\frac{1}{p-1} < b< \frac{p}{p-1}
\end{align}
 and choose $M \in \bN$ such that 
$\frac{Mb}{(p-1)p}, \frac{M}{p-1} \in \bZ$.\footnote{These technical conditions on $b$ and $M$ are used to prove the statement (b) of Theorem \ref{ftheorem}.} 
Then we choose $\pi \in \bC_p$ such that
\begin{align}\label{pisol}
\pi^M = p.
\end{align}
Denote $\mathbb{Q}_q$ be the fraction field of $\bZ_q=W(\mathbb{F}_q)$, the ring of Witt vectors of $\mathbb{F}_q$. Let 
\bea\label{kk}
\Bbbk = \bQ_q(\zeta_p, \pi)
\eea
 be the smallest subfield of $\bC_p$ containing $\zeta_p, \pi$, and $\mathbb{Q}_q$. Denote by $\cO_\Bbbk$ the ring of integers of $\Bbbk$.


For a given polynomial $F(\ud z) \in \bF_q[\ud z]$, denote its Teichm\"uller lifting by $\tilde F(\ud z)$.
 In other words, if we write $F(\ud z) = \sum_{\ud w \in \bNn^N} f_{\ud w} \ud z^{\ud w}$, then
we  have $\tilde F(\ud z) = \sum_{\ud w \in \bNn^N} \tilde f_{\ud w} \ud z^{\ud w} \in \bZ_q[\ud z]$ where 
$(\tilde f_{\ud w})^q = \tilde f_{\ud w}$ 
and $\tilde f_{\ud w} \equiv f_{\ud w} (\modulo p)$.

We review Dwork's splitting function. 
Let $\g \in \bQ_q(\zeta_p) \subset \Bbbk$ be a solution of $\sum_{n=0}^\infty \frac{t^{p^n}}{p^n} =0$ such that $\val_p(\g)=\frac{1}{p-1}$. Following \cite{AdSp08}, we consider 
$$
\hat{\theta}(t) := \prod_{i=0}^{\infty} \theta(t^{p^i}) \quad \text{where} \quad  \theta(t) := E(\g t) = \sum_{i=0}^\infty \l_i t^i \in \Bbbk[[t]],
$$
where $E(t)$ is the Artin-Hasse exponential series
$$
E(t) = \exp \left(\sum_{n=0}^\infty \frac{t^{p^n}}{p^n} \right) \in \Bbbk[[ t ]].
$$
One can easily check that 
\bea \label{thetah}
\hat\theta (t) = \exp \left( \sum_{\ell =0}^\infty \g_\ell t^{p^\ell}   \right) \quad \text{where} \quad \g_\ell:=\sum_{i=0}^\ell \frac{\g^{p^i}}{p^i}.
\eea
Note that 
\begin{align}\label{estimate}
\val (\g_\ell) \geq \frac{p^{\ell+1}}{p-1}-(\ell+1).
\end{align}
For $a \in \bR_{\geq 0}$, we use the following notations
$$
D(a) = \{x \in \bC_p : |x|_p \leq a \}, \quad D(a^-)=\{x \in \bC_p : |x|_p < a \}.
$$
Then $\theta(t)$ converges on $D(p^{\frac{1}{p-1} -})$ (since $\val_p(\l_i) \geq \frac{i}{p-1}$) and $E(t)$ converges only on $D(1^-)$. 
Here $\theta(t)$ is called the Dwork splitting function because of the following lemma of Dwork.
\begin{lemma} \label{chr}
If we define a function $\psi_q:\bF_q \to \bC_p^*$ by the formula
$$
\psi_q(x) := \theta(\tilde x)\cdot  \theta(\tilde x^p) \cdots \theta(\tilde x^{p^{a-1}})
$$
for $x \in \bF_q$ and the Teichm\"uller representative $\tilde x\in \bZ_q$ of $x$, then $\psi_q$ is an additive character of $\bF_q$.
\end{lemma}

\subsection{GBV algebra associated to $\BP^n$}\label{sub4.3}

 We are interested in a $p$-adic homotopy Lie algebra and the zeta function (Theorem \ref{ftheorem}) of a smooth projective complete intersection variety in $\BP^n$ over $\bF_q$ defined by $G_1(\ud x), \cdots, G_k(\ud x)$. In order to study them, we start from $\BP^n$.

For $ \ud u = (u_1,... ,u_{m}) \in \bR^m$, we put $|\ud u |  = u_1 +\cdots+u_{m}$. Let $\bNn$ be the set of non-negative integers.
For each $b$ satisfying \eqref{bcond}, the $p$-adic Banach algebra $\widetilde {A(b)}$ is given as follows:
\begin{eqnarray}\label{pb}
\widetilde {A(b)} = \left\{ \xi(\ud z)=\sum_{ (\ud u, \ud v) \in \bNn^{N}}  a_{ \ud u, \ud v} \pi^{Mb |\ud v|}  \ud x^{\ud u}\ud y^{\ud v} : a_{\ud u, \ud v} \in  \Bbbk \text{ and } a_{\ud u, \ud v} \to 0 \text{ as } |(\ud u, \ud v)| \to \infty  \right\}
\end{eqnarray}
where  $\ud y^{\ud v}= (y_1^{v_1}, \cdots, y_k^{v_k})$ and $\ud x^{\ud u} = (x_0^{u_0}, \cdots, x_n^{u_n})$.
Here the overconvergent factor $\pi^{Mb |\ud v|}$ is specifically designed (in \cite{AdSp08}) to prove the Katz conjecture: $T_q \left( F^i\widetilde {A(b)} \right) \subseteq F^i \widetilde {A(bq)}$ where $T_q$ is given in \eqref{DPSI} and the filtration $F^i$ is given in \eqref{fabs}. The $p$-adic Banach structure on $\widetilde {A(b)}$ is given by $|\xi (\ud z)|_p = \sup_{(\ud u,\ud v)} |a_{\ud u,\ud v}|_p$.

Let $\frg_\Bbbk$ be an abelian Lie algebra over $\Bbbk$ of dimension $N=n+k+1$. Let $\b_{1}, \b_2, \cdots, \b_{N}$ be a ${\Bbbk}$-basis of $\frg_\Bbbk$. 
We associate a Lie algebra representation $\tilde \rho$ on $\widetilde \Ab$ of 
$\frg_\Bbbk$ as follows:
$$
\tilde \rho(\b_i) := \frac{\pi^{Mb}}{\g} \left( \pa{z_i} \right), \text { for $i=1, 2, \cdots, N$}
$$
where we notice that 
$$
\frac{\partial}{\partial z_i} \widetilde {A(b)}\subset \widetilde {A(b)}, \quad i=1,2, \cdots, N.
$$
We extend this ${\Bbbk}$-linearly to
get a Lie algebra representation $\rho_X: \frg_\Bbbk\to\End_{{\Bbbk}}(\widetilde \Ab)$. 
Then we consider \textit{the dual Chevalley-Eilenberg complex} $(\widetilde \cAb, \tilde \Delta)$ associated to $\tilde \rho$ as in the introduction.
The $\Z$-graded super-commutative algebra $\cA_{\tilde \rho}$ and the differential $\tilde \Delta$ is given explicitly as follows\footnote{The auxiliary factor $\frac{\pi^{Mb}}{\gamma}$ is a technical condition needed to prove (b) of Theorem \ref{ftheorem}.}:
\bea \label{pbd}
\begin{split}
\widetilde \cAb&=& \widetilde \Ab[\ud \eta]= \widetilde \Ab[\eta_{1},\eta_2, \cdots, \eta_N], \\
\tilde \Delta&=& \frac{\pi^{Mb}}{\g} \cdot \pa{z_i}\pa{\eta_i}:\widetilde \cAb \to \widetilde \cAb.
\end{split}
\eea
 We have
$$
\xymatrix{0 \ar[r] & \widetilde \cAb^{-N} \ar[r]^-{\tilde \Delta} & \widetilde \cAb^{-N+1} \ar[r]^-{\tilde \Delta} & \cdots \ar[r]^-{\tilde \Delta} & \widetilde \cAb^{-1} \ar[r]^-{\tilde \Delta} & \widetilde \cAb^{0}=\widetilde\Ab \ar[r] & 0}
$$
where
$$
 \widetilde \cAb^{-s} = \bigoplus_{1\leq i_1 < \cdots < i_s\leq N} \widetilde \Ab \cdot \eta_{i_1} \cdots \eta_{i_s}, \quad 0 \leq s \leq N.
$$
\begin{proposition} \label{PBV}
$(\widetilde \cAb,\cdot, \tilde \Delta, \ell_2^{\tilde\Delta})$ is a GBV-algebra over ${\Bbbk}$.
\end{proposition}
\begin{proof}
We need to show that $(\widetilde \cAb, \tilde \Delta, \ell_2^{\tilde\Delta})$ is a dgla and $(\widetilde \cAb, \cdot, \ell_2^{\tilde \Delta})$ is a G-algebra. The fact that $\tilde \Delta$ is a homogeneous differential operator of order 2 implies that $\ell_m^{\tilde \Delta} = 0$ for $m\geq 3$, which implies that $(\widetilde \cAb, \tilde \Delta, \ell_2^{\tilde\Delta})$ is a dgla. It follows that $(\widetilde \cAb, \cdot, \ell_2^{\tilde \Delta})$ is a G-algebra by a straightforward computation.
\end{proof}
%

Let us consider the Dwork operator $T_q:\widetilde \Ab \to \widetilde \Ab$ defined by
\begin{align}\label{DPSI}
T_q\left(\sum_{\ud w \in \bNn^N} a_{\ud w} \ud z^{\ud w}\right) = \sum_{\ud w \in \bNn^N} a_{q\ud w} \ud z^{\ud w}.
\end{align}
Note that the image of $T_q$ belongs to $\widetilde {A(bq)} \subset \widetilde {A(b)}$. By using the relation
\begin{align*}
T_q\circ z_i\frac{\partial}{\partial z_i}=qz_i\frac{\partial}{\partial z_i}\circ T_q,
\end{align*}
one can easily cook up a cochain endomorphism of $(\widetilde \cAb, \tilde \Delta)$: If we define a $\Bbbk$-linear endomorphism $\Psi_{\BP^n}$ on $\widetilde \cAb = \bigoplus_{-N\leq m \leq 0} \widetilde \cAb^m$ by additivity and the following formula
\begin{align}\label{rfrob}
\Psi_{\BP^n}^{-m}(\xi(\ud z) \cdot \eta_{i_1} \cdots \eta_{i_m} ):= q^m \frac{T_q \left( \xi(\ud z) \cdot \prod_{\substack{j \in \{1,\cdots, N\}  \\ j\neq i_1,\cdots, i_m} } z_j\right)}{\prod_{\substack{j \in \{1,\cdots, N\}  \\ j\neq i_1,\cdots, i_m} } z_j}\cdot \eta_{i_1} \cdots \eta_{i_m},
\end{align}
for each $m \geq 0$ and $\xi(\ud z) \in \widetilde \Ab=\widetilde \cAb^0$. 

\begin{lem}
The map $\Psi_{\BP^n}: (\widetilde \cAb, \tilde \Delta) \to (\widetilde \cAb, \tilde \Delta)$ is a cochain map.
\end{lem}

\begin{proof}
The result follows from the computation below:
\begin{align*}
&\quad\left(\Psi_{\mathbf{P}^n}^{-m+1}\circ \tilde \Delta\right)(\xi(\underline{z})\eta_{i_1}\cdots\eta_{i_m})\\
&=\frac{\pi^{Mb}}{\gamma}\cdot\Psi_{\mathbf{P}^n}^{-m+1}\left(\sum_{k=1}^m\frac{\partial\xi(\underline{z})}{\partial z_{i_k}}\eta_{i_1}\cdots\hat{\eta_{i_k}}\cdots\eta_{i_m}\right)\\
&=\frac{\pi^{Mb}}{\gamma}\cdot\sum_{k=1}^m\frac{q^{m-1}}{z_{i_k}\prod_{\substack{j\in\{1,\cdots,N\} \\ j\neq i_1,\cdots,i_m}}z_j}T_q\left(z_{i_k}\frac{\partial}{\partial z_{i_k}}\left(\xi(\underline{z})\prod_{\substack{j\in\{1,\cdots,N\} \\ j\neq i_1,\cdots,i_m}}z_j\right)\eta_{i_1}\cdots\hat{\eta_{i_k}}\cdots\eta_{i_m}\right)\\
&=\frac{\pi^{Mb}}{\gamma}\cdot\sum_{k=1}^m\frac{q^m}{\prod_{\substack{j\in\{1,\cdots,N\} \\ j\neq i_1,\cdots,i_m}}z_j}\frac{\partial}{\partial z_{i_k}}\left(T_q\left(\xi(\underline{z})\prod_{\substack{j\in\{1,\cdots,N\} \\ j\neq i_1,\cdots,i_m}}z_j\right)\frac{\partial}{\partial\eta_{i_k}}(\eta_{i_1}\cdots\eta_{i_m})\right)\\
&=\frac{\pi^{Mb}}{\gamma}\cdot\sum_{i=1}^N\frac{\partial}{\partial z_i}\frac{\partial}{\partial\eta_i}\left(\frac{q^m}{\prod_{\substack{j\in\{1,\cdots,N\} \\ j\neq i_1,\cdots,i_m}}z_j}T_q\left(\xi(\underline{z})\prod_{\substack{j\in\{1,\cdots,N\} \\ j\neq i_1,\cdots,i_m}}z_j\right)\eta_{i_1}\cdots\eta_{i_m}\right)\\
&=\left(\tilde \Delta\circ\Psi_{\mathbf{P}^n}^{-m}\right)(\xi(\underline{z})\eta_{i_1}\cdots\eta_{i_m}).
\end{align*}
\end{proof}

\subsection{dGBV algebra over a $p$-adic field}\label{sub4.4}

We apply a formal deformation theory to the previous subsection \ref{sub4.3} to reveal $p$-adic quantum BV structure for $X$. In particular, we prove the existence of dGBV algebra and the cochain map in Theorem \ref{ftheorem} with property (a).

\begin{definition}  \label{ssdef}
Let us write $G_\ell(\ud x) = \sum_{\ud u \in \bNn^{n+1}} g_{\ell, \ud u} \ud x^{\ud u} \in \bF_q[\ud x]$ for each $\ell =1, \cdots, k$.
Then its Teichm\"uller lift $\tilde G_\ell(\ud x)$ can be written as $ \sum_{\ud u \in \bNn^{n+1}} \tilde g_{\ell, \ud u} \ud x^{\ud u} \in \bZ_q[\ud x]$. We define $p$-adic Dwork potential functions:
\bea
\begin{split} \label{ssdef}
\tilde S(\ud z) &= \sum_{\ud w \in \bNn^N} \tilde s_{\ud w} \ud z^{\ud w}= 
\sum_{\ell=1}^k \sum_{\ud u \in \bNn^{n+1}} \tilde g_{\ell, \ud u}\cdot y_\ell \ud x^{\ud u},& \\
\hat S(\ud z) &=\sum_{m=0}^\infty  \sum_{\ud w \in \bNn^N}\g_m \left(\tilde s_{\ud w} \ud z^{\ud w} \right)^{p^m}
=\sum_{m,\ell,\ud u}  \g_m \cdot y_\ell^{p^m}  (\tilde g_{\ell, \ud u} \ud x^{\ud u})^{p^m} \quad  (\g_m:=\sum_{i=0}^m \frac{\g^{p^i}}{p^i} )&.
\end{split}
\eea
\end{definition}

Following \cite{AdSp08} again, we use the following notation (using \eqref{thetah}):
\bea\label{exp1}
E_{\tilde S}(\ud z) := \prod_{j=0}^{a-1} \prod_{\ud w} \theta ((\tilde s_{\ud w}  \ud z^{\ud w})^{p^j})=e^{\hat S(\ud z)}, 
\quad \hat E_{\tilde S}(\ud z) :=  \prod_{\ud w} \hat \theta ( \tilde s_{\ud w} \ud z^{\ud w})
\eea
so that $E_{\tilde S}(\ud z) = \hat E_{\tilde S}(\ud z) / \hat E_{\tilde S}(\ud z^q)=e^{\hat S(\ud z)}/e^{\hat S(\ud z^q)}$ where $\ud z^q= (z_1^q, \cdots, z_N^q)$.
As operators, we have the following identity:
$$
e^{-\hat{S}} \circ  \frac{\partial}{\partial z_i} \circ e^{\hat{S}} =\hat E_{\tilde S}(\ud z)^{-1} \circ \frac{\rd}{ \rd z_i} \circ \hat E_{\tilde S}(\ud z) =  \frac{\rd}{\rd z_i}+ \frac{\rd \hat S(\ud z)}{\rd z_i}, \quad i =1, \cdots, N=n+k+1.
$$

We define the differential $\tilde K_S$:
\be
\begin{split}
\tilde K_{S}=\cdot \sum_{i=1}^N\left( e^{-\hat S}\circ \frac{\pi^{Mb}}{\g}\pa{z_i}\circ e^{\hat S} \right) \pa{\eta_i}= \frac{\pi^{Mb}}{\g} \cdot \sum_{i=1}^N \left(\prt{\hat S(\ud z)}{z_i} + \pa{z_i}\right) \pa{\eta_i}:\widetilde \cAb\to \widetilde \cAb.
\end{split}
\ee
We also define another differential operator
\be
\tilde Q_{S}=\tilde K_{S}-\tilde \Delta= \frac{\pi^{Mb}}{\g} \cdot \sum_{i=1}^N \prt{\hat S(\ud z)}{z_i} \pa{\eta_i}: \widetilde \cAb \to \widetilde \cAb.
\ee

Now we define a $\Bbbk$-linear operator $T_q^S : \widetilde \Ab \to \widetilde \Ab$ by
$$
T_q^S := T_q \circ E_{\tilde S}(\ud z) = \hat E_{\tilde S}(\ud z)^{-1} \circ T_q \circ \hat E_{\tilde S}(\ud z) =e^{-\hat{S}} \circ  T_q \circ e^{\hat{S}} 
$$
where $T_q^S$ means the multiplication by $E_{\tilde S}(\ud z)$
followed by $T_q$. Since $ b < \frac{p}{p-1}$, $E_{\tilde S}(\ud z) \in \cA\!\left(\frac{b}{q}\right)$ and $T_q^S$ is a well-defined
endomorphism of $\widetilde \Ab$:
$$
\widetilde \Ab \hookrightarrow \cA\!\left(\frac{b}{q}\right)  \ato{E_{\tilde S}(\ud z)} \cA\!\left(\frac{b}{q}\right)  \ato{T_q} \widetilde \Ab.
$$
Then $T_q^S$ is a completely continuous $\Bbbk$-linear operator on $\widetilde \Ab$; see \cite{AdSp08}.

\begin{definition}
We define a completely continuous $\Bbbk$-linear endomorphism $\Psi_S$ on $\widetilde \cAb = \bigoplus_{m \leq 0} \widetilde \cAb^m$ by additivity and the following formula
$$
\Psi_S^{-m}(\xi(\ud z) \cdot \eta_{i_1} \cdots \eta_{i_m} ):= q^m \frac{T_q^S \left( \xi(\ud z) \cdot \prod_{\substack{j \in \{1,\cdots, N\}  \\ j\neq i_1,\cdots, i_m} } z_j\right)}{\prod_{\substack{j \in \{1,\cdots, N\}  \\ j\neq i_1,\cdots, i_m} } z_j}\cdot \eta_{i_1} \cdots \eta_{i_m},
$$
for each $m \geq 0$ and $\xi(\ud z) \in \widetilde \Ab=\widetilde \cAb^0$.
\end{definition}
%

\begin{thm} \label{mainthm}
$(\widetilde \cAb, \cdot, \tilde K_S, \ell^{\tilde K_S}, \tilde Q_S)$ is a $p$-adic dGBV algebra over a $p$-adic field $\Bbbk$ associated to $X_{\ud G}$ and
the map $\Psi_{S}: (\widetilde \cAb, \tilde K_S) \to (\widetilde \cAb, \tilde K_S)$ is a cochain map.
\end{thm}

\begin{proof}

The main tool is to apply the deformation formalism of the subsection \ref{sub3.3} to $(\widetilde \cAb, \cdot, \tilde \Delta)$.
Let $X_{\ud t}\subseteq\mathbf{P}^n$ be a family of smooth projective complete intersections  parametrized by variables $\ud t =t_1, \cdots, t_k$.
For $i =1, \cdots, k$, let $G_i(\ud t, \ud x) \in \bF[\ud t, \ud x]$ be homogeneous polynomials of $\ud x$-degree $d_i$, i.e. 
$$
G_i (\ud t, \ud x) = \sum_{\ud w \in \bZ_{\geq 0}^k}G_{i, \ud w}(\ud x) \ud t^{\ud w} \in \bF_q[\ud t, \ud x], \quad 
G_{i, \ud w}(\ud x) = \sum_{\ud u \in \bZ_{\geq 0}^{n+1}} g_{i,\ud w, \ud u}\ud x^{\ud u} \in \bF_q[\ud x],
$$
where $\deg (G_{i,\ud w}(\ud x)) = d_i$, such that $G_i(\ud x) = G_i (\ud 1, \ud x)$ and $G_i(\ud 0, \ud x) =0$.
Let $\tilde G_i(\ud t, \ud x)$ be the
 Teichm\"uller lifting:
\begin{align*}
\tilde G_i(\ud t, \ud x)=\sum_{\underline{u}\in\mathbb{Z}_{\geq 0}^{n+1}, \ud w \in \bZ_{\geq 0}^k}\tilde{g}_{i,\underline{u}, \ud w}\underline{x}^{\underline{u}} \ud t^{\ud w}
\end{align*}
where $\tilde{g}_{i,\underline{u}, \ud w} \in \bZ_q$ is the Teichm\"uller lifting of $g_{i,\underline{u}, \ud w} \in \bF_q$.
Let $S(\ud t, \ud z)=\sum_{i=1}^k y_i G_i(\ud t, \ud x)$ and
let $\tilde S(\ud t, \ud z)$ be the Teichm\"uller lifting of $S(\ud t, \ud z)$:
$$
\tilde S(\ud t, \ud z) = \sum_{i =1}^k y_i \tilde G_i(\ud t, \ud x) =
\sum_{i,\ud u, \ud w}y_i \tilde{g}_{i,\underline{u}, \ud w}\underline{x}^{\underline{u}} \ud t^{\ud w}
\in \bZ_q[\ud t, \ud z]=\bZ_q[\ud t, \ud x, \ud y],
$$
where $\bZ_q$ is the ring of integers of $\bQ_q$.

Let $\frak{a} = \Bbbk[[t_1,\cdots,t_k]] \in\hbox{Ob}(\widehat{\art_{\Bbbk}} )$ with degree $|t_i|=0$.
Since $\widetilde \cAb^{1}=0$, the Maurer-Cartan equation $\tilde \Delta(e^\Gamma)=0$ is vacuous, i.e. we can deform $(\widetilde {\cAb}, \cdot, \tilde \Delta)$ formally by any element $\tilde \G$ in $(\mathfrak{m_a}\otimes\Bbbk[[ \ud z ]])^0$. 
We take
\begin{align}\label{sgm}
\xymatrix{\sigma:\Bbbk[[t_1,\cdots,t_k]] \ar[r] & \Bbbk[[t_1,\cdots,t_k]] & t_i \ar@{|->}[r] & t_i^q}
\end{align}
and
\begin{align}\label{onef}
\tilde \Gamma_{\ud t}:=\sum_{i,\underline{u}, \ud w}\sum_{\ell\geq0}\gamma_\ell y_i^{p^\ell} \tilde{g}_{i,\underline{u}, \ud w}^{p^\ell}\underline{x}^{p^\ell\underline{u}}\ud t^{p^\ell \ud w}
=\sum_{i,\underline{u}, \ud w}\sum_{\ell\geq0}\gamma_\ell \left( y_i \tilde{g}_{i,\underline{u}, \ud w}\underline{x}^{\underline{u}}\ud t^{ \ud w}\right)^{p^\ell}
\in((\underline{t})\otimes \Bbbk[[\ud z]])^0,
\end{align}
and consider
\begin{align*}
e^{\tilde \Gamma_{\ud t}}
&=\exp\left(\sum_{i,\underline{u}, \ud w}\sum_{\ell\geq0}\gamma_\ell \left( y_i \tilde{g}_{i,\underline{u}, \ud w}\underline{x}^{\underline{u}}\ud t^{ \ud w}\right)^{p^\ell}\right) 
=\prod_{i,\underline{u}, \ud w}\exp\left(\sum_{\ell\geq0}\gamma_\ell \left( y_i \tilde{g}_{i,\underline{u}, \ud w}\underline{x}^{\underline{u}}\ud t^{ \ud w}\right)^{p^\ell}\right)=\prod_{i,\underline{u}, \ud w} \prod_{\ell\geq0}\hat \theta\left( y_i \tilde{g}_{i,\underline{u}, \ud w}\underline{x}^{\underline{u}}\ud t^{ \ud w}\right)
\end{align*}
where $\hat \theta$ is given in \eqref{thetah}.
Note that neither $\tilde \Gamma_{\ud t}$ nor $e^{\tilde \Gamma_{\ud t}}$ belong to
$\mathfrak{m_a} \otimes \widetilde {A(b)}=(\mathfrak{m_a} \otimes \widetilde \cAb)^0$.
But the formal deformation of $\tilde \Delta$ by $\tilde \Gamma_{\ud t}$ (note that $\tilde \Delta$ is $\fra$-linear) can be computed as follows:
\begin{align*}
\tilde \Delta_{\tilde \Gamma_{\ud t}}:=
e^{-\tilde \Gamma_{\ud t}} \circ \tilde \Delta \circ e^{\tilde \Gamma_{\ud t}}=\frac{\pi^{Mb}}{\gamma}\cdot\sum_{i=1}^N\left(\frac{\partial\hat{S}(\ud t, \underline{z})}{\partial z_i}+\frac{\partial}{\partial z_i}\right)\frac{\partial}{\partial\eta_i},
\end{align*}
(the equalities here are formal ones in $\fra \otimes \Bbbk[[\ud z]]$ without considering $p$-adic convergence) 
where 
$$
\frac{\partial\hat{S}(\ud t, \underline{z})}{\partial z_i}
 =\frac{\partial}{\partial z_i} \left(\tilde \G|_{\ud t =\ud 1}\right), \quad i=1, \cdots, N,
$$
 by \cite[Lemma 3.1]{PP} and the fact $\ell_m^{\tilde \Delta}=0, \ m\geq 3$.
Then the expression $\frac{\pi^{Mb}}{\gamma}\cdot\sum_{i=1}^N\left(\frac{\partial\hat{S}(\underline{t},\underline{z})}{\partial z_i}+\frac{\partial}{\partial z_i}\right)\frac{\partial}{\partial\eta_i}$
actually makes sense as a $\Bbbk$-linear operator of $\fra \otimes \widetilde \cAb$, since 
$$
\frac{\partial\hat{S}(\underline{t},\underline{z})}{\partial z_i} \in (\ma \otimes \widetilde \cAb)^0, \quad i=1, \cdots, N,
$$
by using the estimate (\ref{estimate}).

To compute $\Psi_{\BP^n, \sigma,\tilde \Gamma_{\ud t}}$ (the deformation of $\Psi_{\BP^n}$ by $\sigma$ and $\tilde \Gamma_{\ud t}=\tilde \Gamma_{\ud t}(\ud z)$), observe that
\begin{align*}
E(\underline{z},\underline{t}^q)\circ T_q=T_q\circ E(\underline{z}^q,\underline{t}) \quad \text{in }  \ma \otimes \Bbbk[[\ud z]]
\end{align*}
for any power series $E(\underline{z},\underline{t})$ so formally (note that $\Psi_{\BP^n}$ is $\fra$-linear)
\begin{align*}
\Psi_{\BP^n, \sigma,\tilde \Gamma_{\ud t}}
=e^{-\sigma(\tilde \Gamma_{\ud t})} \circ \Psi_{\BP^n} \circ e^{\tilde \Gamma_{\ud t}}
=\Psi_{\BP^n}\circ \frac{e^{\tilde \Gamma_{\ud t}(\ud z)}}{e^{\tilde \Gamma_{\ud t}(\ud z^q)}}
=\Psi_{\BP^n}\circ e^{\Gamma_{\ud t}},
\end{align*}
where 
\begin{align}\label{twof}
\Gamma_{\ud t}=\Gamma_{\ud t}(\ud z):=\sum_{j=0}^{a-1}\sum_{i,\underline{u}, \ud w}\sum_{\ell\geq0}\frac{ \gamma^{p^\ell} \left(\tilde{g}_{i,\underline{u}}y_i\underline{x}^{\underline{u}} \ud t^{\ud w} \right)^{p^{j+\ell}}}{p^{\ell}}.
\end{align}
Note that the last (formal) equality $\frac{e^{\tilde \Gamma_{\ud t}(\ud z)}}{e^{\tilde \Gamma_{\ud t}(\ud z^q)}}
=e^{\Gamma_{\ud t}}$ crucially uses the fact $\tilde G_i (\ud t, \ud x)$ is the Teichm\"uller lifting of $G_i(\ud t, \ud x)$:
$\tilde{g}_{i,\underline{u}, \ud w}^q= \tilde{g}_{i,\underline{u}, \ud w}$.
The key point here is that the expression $\Psi_{\BP^n}\circ e^{\G_{\ud t}}$ makes sense\footnote{If we use $\Psi_{\BP^n, \tilde \Gamma_{\ud t}}$ instead of $\Psi_{\BP^n, \sigma,\tilde \Gamma_{\ud t}}$, then we can not achieve this.} as a $\Bbbk$-linear operator of $\fra \otimes \widetilde \cAb$, since  $e^{\G_{\ud t}}$ belongs to $\mathfrak{m_a} \otimes \widetilde {A(\frac{b}{q})}$, $\Psi_{\BP^n}: \widetilde {A(\frac{b}{q})} \to  \widetilde {A(b)},$ and $ \widetilde {A(b)} \subset  \widetilde {A(\frac{b}{q})}$.
%
By evaluating at $\ud t = \ud 1$, one can directly check that
\begin{align}\label{xdef}
\begin{split}
\tilde K_S&:=\left.\tilde \Delta_{\tilde \G_{\ud t}}\right|_{\ud t=\ud 1} \\
\Psi_S&:= \left.\Psi_{\BP^n,\sigma, \tilde \G_{\ud t}}\right|_{\ud t=\ud 1}.
\end{split}
\end{align}

By the formal identity in the deformation theory, we conclude that $(\widetilde \cAb, \tilde K_S)$ is a cochain complex. $\Psi_S: (\widetilde \cAb, \tilde K_S) \to (\widetilde \cAb, \tilde K_S)$ is a cochain map. Now the part $(a)$ follows directly by applying the construction of Definition \ref{defdesc} to $\Psi_S$.
\end{proof}

\begin{proof}[A proof of (a) of Theorem \ref{ftheorem}]
The existence of $L_\infty$-endomorphism follows directly by applying the construction of Definition \ref{defdesc} to the cochain map $\Psi_S: (\widetilde \cAb, \tilde K_S) \to (\widetilde \cAb, \tilde K_S)$. The existence of the decomposition is clear from the construction of $\widetilde \cAb$.
\end{proof}

\subsection{$p$-adic $0$-dimensional quantum field theory and BV formalism}\label{sub4.5}

We briefly explain the physical interpretation ($p$-adic quantum BV formalism) behind the $p$-adic construction. We decided to include it because the physical viewpoint was crucial to the conception of the paper. 
Even if the $p$-adic description is not entirely precise\footnote{A $p$-adic measure (or distribution) theory and $p$-adic rigid geometry need a more careful analysis.} from a mathematical point of view, our hope is that it will be more helpful than confusing in guiding the reader through the rather elaborate constructions in the article.

%

%


We consider the polydisc
\be
D_N(p^{b|\ud v|}):=\{ \ud z \in \Bbbk^N \ : \ |z_i|_p\leq p^{b|\ud v|}, \ i=1, \cdots, N\}.
\ee
Then the elements of $\widetilde {A(b)}$ can be viewed as $p$-adic analytic functions on the polydisc $D_N(p^{b|\ud v|})$.
Our $p$-adic 0-dimensional field theory $\QFT^\Bbbk_S$ has $N$-number of points as space-time and $D_N(p^{b|\ud v|})$ as space of fields.

\begin{prop}
The $p$-adic dGBV algebra $(\widetilde \cAb, \cdot, \tilde K_S, \ell_2^{\tilde K_S}, \tilde Q_S)$
provides quantum BV formalism\footnote{The Berezinian measure in Definition \ref{bv2} does not make sense, since there is no $p$-adic Haar measure ($p$-adically bounded) on $\bQ_p^N$ or $\Bbbk^N$. But the Berezinian measure part can be conceptually replaced by a cochain map $(\widetilde \cAb, \tilde K_S) \to (\Bbbk,0)$ or a cochain endomorphism of $(\widetilde \cAb, \tilde K_S)$. } for $X$:
\begin{itemize}
\item the (-1)-shifted cotangent bundle $T^*[-1]\left(D_N(p^{b|\ud v|})\right)$ over $\Bbbk$, whose (global section of) sheaf of $p$-adic analytic functions is given by $\widetilde \cAb$
\item  odd symplectic form $\o=\sum_{i=1}^N d\eta_i \wedge d z_i$
 and the associated Poisson bracket\footnote{One can show that $(-1)^{|f|}\ell_2^{\tilde K_S}(f,g)=(-1)^{|f|}\ell_2^{\tilde \Delta}(f,g)=\frac{\pi^{Mb}}{\g} \{ f, g\}.$} is given by 
\be
\{ f, g\}:= \sum_{i=1}^N \frac{\partial f}{\partial z_i} \frac{\partial g}{\partial \eta_i} - (-1)^{|f|+1} \frac{\partial f}{\partial \eta_i} \frac{\partial g}{\partial z_i}, \quad f, g \in \cA.
\ee
\item the BV action functional $\hat S$ which satisfies the quantum master equation\footnote{Note that $\tilde K_S (f) = \frac{\pi^{Mb}}{\g} \{\hat S, f \} + \frac{\pi^{Mb}}{\g} \Delta(f)$.}
\be
\{\hat S, \hat S \} + \Delta (\hat S) = 0 \iff \tilde K_S(\hat S) =0
\ee

\end{itemize}
\end{prop}
\begin{proof}
The definitions \ref{bv1} and \ref{bv2} basically make sense over $\Bbbk$ except the complex number $i$, the formal parameter $\hbar$, and the Berezinian measure part, since the definitions are algebraic. 
$T^*[-1]\left(D_N(p^{b|\ud v|})\right)$ is a $p$-adic supermanifold whose algebraic structure sheaf is given by $\widetilde \cAb$.
Using a Darboux coordinate $\ud \eta=(\eta_1, \cdots, \eta_N), \ud z=(z_1, \cdots, z_N)$ of $\cF$, $\o=\sum_{i=1}^N d\eta_i \wedge d z_i$ defines an odd symplectic structure on $T^*[-1]\left(D_N(p^{b|\ud v|})\right)$.
The BV Laplacian associated to the Darboux coordinate is given by
\be
\Delta:=\sum_{i=1}^N\frac{\partial}{\partial z_i}\frac{\partial}{\partial\eta_i}.
\ee
\end{proof}

The uncertainty principle of quantum field theory is described by the following equality
\bea \label{auc}
\hbar z_i \circ \frac{\partial}{\partial z_i}= \frac{\partial}{\partial z_i} \circ \hbar  z_i + \hbar
\eea
where $z_i$ is ``the position operator'' and $\frac{\partial}{\partial z_i}$ is ``the momentum operator''.\footnote{This is equivalent to \eqref{elltwo} which is the failure of the Leibniz rule of the BV operator $\Delta_{\mu_0}$.} These operators $z_i$ and $\frac{\partial}{\partial z_i}$ are key examples of symmetries\footnote{These operators act on $A_{\BP(\cE)}$.} in $\QFT_S^{\bC}$. According to the philosophy of Noether's principle in physics, there should be some invariants associated to these symmetries; the middle-dimensional cohomology $H_{dR, \pr}^{n-k}(X, \bC)$ and its period integrals in \eqref{fpi} and period matrices are such examples of invariants.\footnote{We do not delve into a precise meaning of invariants, since we only deal with physical ideas behind the paper.} In a non-archimedean quantum field theory, there is a new symmetry, namely the $p$-power Frobenius map or its left inverse $T_q$ in \eqref{DPSI}, which should induce a new interesting invariant. In our example $\operatorname{QFT}_{S}^\Bbbk$, this invariant is the zeta function associated to $X$ or the number of $\bF_{q^m}$-rational points of $X$ for each $m\geq 1$. A non-archimedean version of the uncertainty principle of quantum field theory can be described by
\bea\label{nuc}
T_q\circ z_i\frac{\partial}{\partial z_i}=qz_i\frac{\partial}{\partial z_i}\circ T_q.
\eea

If we adapt the analogy between $\hbar$ in archimedean theory and $q=p^a$ in non-archimedean theory, the $p$-adic 0-dimensional field theory $\QFT^{\Bbbk}_S$, in some sense (Definition \ref{bv3}), can be viewed as a BV quantization of classical BV formalism over $\bF_q$; we refer to (b) of Theorem \ref{ftheorem} and subsection \ref{sub4.1}.

\begin{table}[ht]
	\caption{0-dimensional field theory over $\bC, \bF_q,$ and  $\Bbbk$}
	\begin{tabular}{|c|c|c|c|}
	\hline
	&&&\\[-1.5ex]
	space of fields & $\bC^N$ & $\bF_q^N$ & $D_N(p^{b|\ud v|}) \subset \Bbbk^N$ \\[1ex]
	\cline{1-1}\cline{2-2}\cline{3-3}\cline{4-4}
	&&&\\[-1ex]
	action functional & $S(\ud z) \in A_{\BP(\cE)}=\bC[\ud z]$ & $S(\ud z) \in A=\bF_q[\ud z] $ & $\hat S(\ud z) \in \widetilde{A(b)}$\\[2ex]
	\hline
	&&&\\[-1ex]
	symmetries & $z_i$ and $\frac{\partial}{\partial z_i}$ & $z_i \mapsto z_i^q$ & Dwork operator $T_q$\\[2ex]
	\hline
	&&&\\[-1ex]
	space of BV fields & $T^*[-1](\bC^N)$ &$T^*[-1](\bF_q^N)$ & $T^*[-1]\left(D_N(p^{b|\ud v|})\right)$\\[2ex]
	\hline
	&&&\\[-1ex]
	quantum invariants & period integrals $\int_\g \o$ &$|X_{\ud G}(\bF_{q^m})|=\frac{|Y_{\ud G}(\bF_{q^m})|-1}{q^m-1}$ & $p$-adic Dwork operator $\Psi_S$\\[2ex]
	\hline
	\end{tabular}
	\label{Table1}
\end{table}



Let $Y_{\ud G}$ be the affine complete intersection in $\bA^N$ defined by $G_1(\ud x), \cdots, G_k(\ud x)$.
Let $\psi_q:\bF_q \to \bC_p^*$ be an additive character (see Lemma \ref{chr}). Let $Y_{\ud G}$ be the affine complete intersection in $\bA^N$ defined by $G_1(\ud x), \cdots, G_k(\ud x)$.
Then one can show that
\begin{eqnarray} \label{esum}
\sum_{\ud z \in \bA^N(\bF_q^m)} \psi_q(\Tr_{q^m/q} S(\ud z)) = |Y_{\ud G}(\bF_{q^m})| \cdot q^{mk}
\end{eqnarray}
where $\Tr_{q^m/q}:\bF_{q^m} \to \bF_q$ is the trace map. We have $|X_{\ud G}(\bF_{q^m})|=\frac{|Y_{\ud G}(\bF_{q^m})|-1}{q^m-1}$. Notice that the logarithms of the zeta functions
\be
\log \left( Z(Y_{\ud G}; T)\right)&=&\sum_{m\geq0}|Y_{\ud G}(\bF_{q^m})|\frac{T^m}{m}=
\sum_{m\geq0} \left(\sum_{\ud z \in \bA^N(\bF_{q^m})} \psi_q(\Tr_{q^m/q} S(\ud z))  \right)\frac{T^m}{m}, \quad \\
\log \left( Z(X_{\ud G}; T)\right)&=&\sum_{m\geq0}|X_{\ud G}(\bF_{q^m})|\frac{T^m}{m}
\ee
have the shape of ``Feynman path integral'' of the action functional $S(\ud z) \in \bF_q$.
The characteristic polynomial of the cochain map $\Psi_S$ of $(\widetilde \cAb, \tilde K_S)$ gives us the zeta function; part (d) of Theorem \ref{ftheorem}.
%

\section{Homotopy Lie deformation formula for the zeta function}\label{sec5}

\subsection{$\pi$-adic filtered complex}\label{sub5.1}

Here we prove (b) of Theorem \ref{ftheorem}.
For this we need to understand a precise relationship between $(\widetilde \cAb, \cdot, \tilde \Delta)$ and $(\cA, \cdot, \Delta)$. Let us define a filtration on $(\widetilde \cAb, \cdot, \tilde \Delta)$ which is compatible with the differential $\tilde \Delta$. The technical conditions on $b$ and $M$ and the factor $\frac{\pi^{Mb}}{\g}$ is designed to accomplish this by Adolphson and Sperber in \cite{AdSp08}.
Recall that $\pi$ is a uniformizer for $\cO_\Bbbk$. Following section 3, \cite{AdSp08}, define a decreasing filtration $\{F^s \widetilde \cAb\}_{s\in \bZ}$ on $\widetilde \cAb$:
\begin{eqnarray}\label{fil}
F^s \widetilde \cAb^{-m} =\bigoplus_{r+t =m} \bigoplus_{\substack{1\leq i_1 < \cdots < i_r \leq k \\ k+1 \leq j_1 < \cdots < j_t \leq N }} \pi^{Mb (k-r)} \tilde F^s \widetilde{\Ab} \cdot \eta_{i_1} \cdots \eta_{i_r}\cdot \eta_{j_1} \cdots \eta_{j_t}, \quad 0 \leq m \leq N,
\end{eqnarray}
where 
\begin{eqnarray*}
\tilde F^s\widetilde{\Ab} =\left\{ \sum_{ (\ud u, \ud v) \in \bNn^{N}}  a_{ \ud u, \ud v} \pi^{Mb |\ud v|} \ud   x^{\ud u}\ud  y^{\ud v}: a_{\ud u, \ud v} \in  \pi^s\cO_\Bbbk \text{ for all } (\ud u, \ud v) \in \bNn^{N}  \right\}.
\end{eqnarray*}
Note that
\bea \label{fabs}
F^s\widetilde{\Ab} =F^s\widetilde \cAb^{0}=\left\{ \sum_{ (\ud u, \ud v) \in \bNn^{N}}  a_{ \ud u, \ud v} \pi^{Mb( |\ud v| + k)} \ud x^{\ud u}  \ud y^{\ud v}: a_{\ud u, \ud v} \in  \pi^s\cO_\Bbbk \text{ for all } (\ud u, \ud v) \in \bNn^{N}  \right\}.
\eea

Then a simple calculation confirms that 
$\tilde \Delta(F^s \widetilde \cAb^{-m}) \subset F^s \widetilde \cAb^{-m+1}$,
using the fact $ \frac{1}{p-1} < b < \frac{p}{p-1}$.
(here the factor $\frac{\pi^{Mb}}{\gamma}$ in $\tilde \Delta=\frac{\pi^{Mb}}{\gamma}\cdot\sum_{i=1}^N\frac{\partial}{\partial z_i}\frac{\partial}{\partial\eta_i}$ plays a role). Therefore the filtration (\ref{fil}) makes $(\widetilde \cAb, \tilde \Delta)$ into a $\pi$-adic filtered complex.

For each $0 \leq m \leq N$, define a $\Bbbk$-linear map $\cR: F^0\widetilde \cAb^{-m} \to \cA^{-m}$, where $\cA$ was given in (\ref{finite}), by additivity and the formula:
\begin{eqnarray*}
\sum_{ (\ud u, \ud v) \in \bNn^{N}}  a_{ \ud u, \ud v} \pi^{Mb (|\ud v|+k-r)} \ud x^{\ud u}  \ud y^{\ud v} 
\cdot \eta_{i_1} \cdots \eta_{i_r}\cdot \eta_{j_1} \cdots \eta_{j_t}
  \mapsto \sum_{ (\ud u, \ud v) \in \bNn^{N}} \overline a_{ \ud u, \ud v} \ud x^{\ud u}  \ud y^{\ud v} 
\cdot \eta_{i_1} \cdots \eta_{i_r}\cdot \eta_{j_1} \cdots \eta_{j_t},
\end{eqnarray*}
where $\overline  a_{ \ud u, \ud v}$ is the reduction of $ a_{ \ud u, \ud v}$ modulo the maximal ideal of $\cO_\Bbbk$.
Since $ a_{\ud u, \ud v} \to 0$ as $|(\ud u, \ud v)| \to \infty$, the image $\cR(\xi)$ for $\xi \in \widetilde \cAb^{-m}$ is a finite sum.
It is not difficult to see that this $\Bbbk$-linear map is surjective with kernel $F^1\widetilde \cAb^{-m}$ (using the conditions $\frac{1}{p-1} < b$ and $\frac{Mb}{(p-1)p}, \frac{M}{p-1} \in \bZ$), hence $\cR$ induces a linear isomorphism
$$
\cR: F^0\widetilde \cAb^{-m}/ F^1\widetilde \cAb^{-m} \simeq \cA^{-m},
$$
for each  $0 \leq m \leq N$. 
Note that this map $\cR$ is not a ring homomorphism.
We can choose a $\Bbbk$-linear section $s_\cR: \cA \to F^0\widetilde \cAb$ such that $\cR \circ s_\cR = \id$.

\begin{proposition}\label{fprop}
The $\Bbbk$-linear map $\cR$ induces an isomorphism of cochain complexes over $\bF_q$
$$
\cR: (F^0\widetilde \cAb/ F^1\widetilde \cAb, \tilde \Delta) \simeq (\cA, 0),
$$
where $(\cA, 0)$ is the cochain complex $\cA$ with the zero differential.
\end{proposition}

\begin{proof}[A proof of (b) of Theorem \ref{ftheorem}]

The part $(b)$ follows by using the decreasing filtration and the $\Bbbk$-linear map $\cR$ in Proposition \ref{fprop}.
A same computation confirms that
\begin{align}\label{isot}
 (F^0\widetilde \cAb/ F^1\widetilde \cAb, \tilde K_S) \xrightarrow{\cR} (\cA, Q_S),
\end{align}
where $(\cA,Q_S)$ is given in (\ref{finite}), 
is an isomorphism of cochain complexes over $\bF_q$.
\end{proof}

\subsection{The computation of cohomology}\label{sub5.2}

Here we prove (c) of Theorem \ref{ftheorem}.
For this we will compare our BV algebra $(\widetilde \cAb, \cdot, K_S)$ with a twisted de Rham complex, so called the $p$-adic Dwork complex, $(\O_b^\bullet, \wedge, D)$ which appeared in section 2, \cite{AdSp08}.
We briefly review the Dwork construction closely following section 2, \cite{AdSp08}.
The degree $m$-th module of $\O_b^\bullet$ is given by
$$
\O_b^m :=  \bigoplus_{1\leq i_1 < \cdots < i_m\leq N} \widetilde \Ab \wedge dz_{i_1}\wedge \cdots \wedge dz_{i_m}
$$
for each $m \geq 0$. Here $\wedge$ is the wedge product on the twisted de Rham complex $\O_b^\bullet$ of $\widetilde \Ab$.
Note that $\O_b^0 = \widetilde \Ab$ and $\O_b^\bullet = \bigoplus_{0 \leq m \leq N} \O_b^m$.
The differential $D$ is defined by
$$
D(w) = \frac{\pi^{Mb}}{\g}   \left( dw +d \hat S(\ud z) \wedge w \right), \quad w\in \O_b^m
$$
for any $m \geq 0$.

\begin{proposition} \label{clem}
We have the following relationship between $(\O_b^\bullet, D)$ and $(\widetilde \cAb, \tilde K_S)$;

(a) For each $s \in \bZ$, if we define a $\Bbbk$-linear map $J: (\O^s_{b}, D) \to (\widetilde \cAb^{s-N}, \tilde K_S)$ by 
$$
dz_{i_1} \cdots d z_{i_s} \mapsto (-1)^{i_1 + \cdots +i_s - s} ( \cdots \hat{\eta_{i_1}} \cdots \hat{\eta_{i_s}}\cdots)
$$ 
for $1 \leq i_1 < \cdots < i_s \leq N=n+k+1$ and extending
it $\Bbbk$-linearly, then $J \circ D = \tilde K_S \circ J$ and $J$ induces an isomorphism
$$
H^s(\O^\bullet_b,D) \simeq H^{s-N} (\widetilde \cAb,\tilde K_S).
$$
for every $s \in \bZ$.

(b) The map $J$ satisfies that  $J \circ \partial_X = \tilde Q_S \circ J$, 
where $\partial_X$ is the wedge product with $\frac{\pi^{Mb}}{\g} d \hat S(\ud z)$.
\end{proposition}

\begin{proof}
These follow from direct computations; this is a version of the odd (fiberwise) Fourier transform in \cite[Proof of Theorem 4.4.12]{Mnev}.
\end{proof}

\begin{proof}[A proof of (c) of Theorem \ref{ftheorem}]
By Proposition \ref{clem}, 
$$
H^N(\Omega_b^\bullet, D) \simeq H^0(\widetilde \cAb, \tilde K_S).
$$
Because the $\Bbbk$-dimension of $H^N(\Omega_b^\bullet, D)$ is shown to be the degree of $P(T)$ in \cite{AdSp08},
we conclude that the degree of $P(T)$ is equal to the $\Bbbk$-dimension of $H^0(\widetilde \cAb, \tilde K_S)$.
\end{proof}

\begin{remark}
Proposition \ref{clem} implies that two cochain complexes $(\O_b^\bullet, D)$ and $(\widetilde \cAb, \tilde K_S)$ are degree-twisted isomorphic to each other.
But we emphasize that the natural product structure, the wedge product, on $\O_b^\bullet$ and  the super-commutative product $\cdot$ on $\widetilde \cAb$
are quite different and $J$ is \textit{not} a ring isomorphism. It is crucial for us to use the super-commutative product $\cdot $ on 
$\widetilde \cAb$ to get the main theorems of this article.
\end{remark}
%
%

\subsection{Zeta functions and $p$-adic Dwork operators} \label{sub5.3}

We will prove (d) of Theorem \ref{ftheorem} this subsection.

\begin{proof}[A proof of (d) of Theorem \ref{ftheorem}]
First note that (using the function in \eqref{exp1})
\begin{align*}
e^{\tilde \Gamma_{\ud t}}
=\prod_{i,\underline{u}, \ud w}\hat{\theta}\left(y_i\tilde{g}_{i,\underline{u}, \ud w}\underline{x}^{\underline{u}}\ud t^{\ud w} \right)=\hat{E}_{\tilde S(\ud t, \ud z)}(\underline{z}), 
\quad e^{\sigma(\tilde \Gamma_{\ud t})}=\hat{E}_{\tilde S(\ud t^q, \ud z)}(\underline{z}),
\end{align*}
where $\sigma$ is given in \eqref{sgm} and $\tilde \Gamma_{\ud t}$ is given in \eqref{onef}.
Similarly, we have that
\begin{align*}
e^{\Gamma_{\ud t}}
=E_{\tilde S(\ud t, \ud z)}(\underline{z}),
\end{align*}
where $\Gamma_{\ud t}$ is given in \eqref{twof}.
Because
\begin{align*}
\tilde \Delta_{\tilde \Gamma_{\ud t}}:=
e^{-\tilde \Gamma_{\ud t}} \circ \tilde \Delta \circ e^{\tilde \Gamma_{\ud t}}=
\hat{E}_{\tilde S(\ud t, \ud z)}(\underline{z})^{-1}\circ \tilde \Delta \circ\hat{E}_{\tilde S(\ud t, \ud z)}(\underline{z})
=\frac{\pi^{Mb}}{\gamma}\cdot\sum_{i=1}^N\left(\frac{\partial\hat{S}(\ud t, \underline{z})}{\partial z_i}+\frac{\partial}{\partial z_i}\right)\frac{\partial}{\partial\eta_i},
\end{align*}
\begin{align*}
\Psi_{\BP^n, \sigma,\tilde \Gamma_{\ud t}}
=e^{-\sigma(\tilde \Gamma_{\ud t})} \circ \Psi_{\BP^n} \circ e^{\tilde \Gamma_{\ud t}}
=\hat{E}_{\tilde S(\ud t^q, \ud z)}(\underline{z})^{-1}\circ \Psi_{\BP^n} \circ\hat{E}_{\tilde S(\ud t, \ud z)}(\underline{z})
=\Psi_{\BP^n}\circ E_{\tilde S(\ud t, \ud z)}(\underline{z}),
\end{align*}
we see that 
the operator 
$\alpha_{n+k+1}$ on $H^{n+k+1}(\O_b^\bullet, D)$, in Corollary 6.5, \cite{AdSp08}, corresponds exactly (by (\ref{xdef})) to the operator $\Psi_S$ on $H^0(\widetilde \cAb,\tilde K_S)$ under the isomorphism $H^{n+k+1}(\O_b^\bullet, D) \simeq H^0(\widetilde \cAb,\tilde K_S)$ in Proposition \ref{clem}.
Then the part $(d)$ of Theorem \ref{ftheorem} 
\begin{eqnarray}
P(q^k T)= \det(1 - T \cdot \Psi_S \big|H^0(\widetilde \cAb, \tilde K_S))
\end{eqnarray}
follows from Corollary 6.5, \cite{AdSp08}.
%
\end{proof}

\subsection{The Bell polynomials and a deformation formula for the Dwork operator $\Psi_S$ }
\label{sub5.4}

As an application of our theory, we provide an $L_\infty$-homotopy deformation formula for $\Psi_S$ using Definition \ref{defdesc}, the Bell polynomials, and some Maurer-Cartan solutions $\hat S$ and $\G$.

The \emph{Bell polynomials} $B_n(x_1,\cdots,x_n)$ are defined by the power series expansion
\begin{align}\label{Bpoly}
\exp\left(\sum_{i\geq1}x_i\frac{t^i}{i!}\right)=1+\sum_{n\geq1}B_n(x_1,\cdots,x_n)\frac{t^n}{n!}.
\end{align}
For example, $B_1(x_1)=x_1, B_2(x_1, x_2)=x_1^2+x_2$, $B_3(x_1,x_2,x_3)=x_1^3+3x_1x_2+x_3, \cdots$.

The deformation theory based on the Maurer-Cartan equation attached to $(\cAb, \widetilde \Delta, \ell_2^{\tilde \Delta})$ naturally leads to the following formula.

\begin{theorem} \label{stheorem} 
Let (recall Definition \ref{ssdef})
\begin{align*}
\Gamma=\sum_{j=0}^{a-1}\sum_{i,\underline{u}}\sum_{\ell\geq0}\frac{ \gamma^{p^\ell} \left(\tilde{g}_{i,\underline{u}}y_i\underline{x}^{\underline{u}}\right)^{p^{j+\ell}}}{p^{\ell}}.
\end{align*}
For any homogeneous $\lambda\in\widetilde \cAb$, we have
\begin{align*}
\tilde K_S(\lambda)&=\tilde \Delta(\lambda) + \sum_{m\geq 1}\frac{1}{m!}\ell_{m+1}^{\tilde \Delta} (\hat S, \cdots, \hat S, \lambda)=\tilde \Delta(\lambda)+\ell_2^{\tilde \Delta}(\hat S, \lambda)=\frac{\pi^{Mb}}{\gamma}\cdot\sum_{i=1}^N\left(\frac{\partial\hat{S}( \underline{z})}{\partial z_i}+\frac{\partial}{\partial z_i}\right)\frac{\partial}{\partial\eta_i}(\l),\\
\Psi_{S}(\lambda)&=\Psi_{\BP^n}(\lambda)+\sum_{m\geq1}\sum_{\substack{j+k=m \\ j,k\geq0}}\frac{1}{j!k!} B_j(\phi^{\Psi_{\BP^n}}_1(\Gamma),\cdots,\phi^{\Psi_{\BP^n}}_j(\Gamma, \cdots, \Gamma))\cdot\phi^{\Psi_{\BP^n}}_{k+1}(\Gamma, \cdots, \Gamma,\lambda),
\end{align*}
where $B_n(x_1,\cdots,x_n)$ is the Bell polynomial defined in (\ref{Bpoly}), and $\phi^{\Psi_S}_1(\Gamma,\lambda):=\phi^{\Psi_{\BP^n}}_1(\lambda)=\Psi_{\BP^n}(\lambda)$.
\end{theorem}

\begin{proof}
Recall that (see \eqref{kg} and \eqref{psg})
\begin{align*}
\begin{split}
\tilde K_S&:=\left.\tilde \Delta_{\tilde \G_{\ud t}}\right|_{\ud t=\ud 1}
=\frac{\pi^{Mb}}{\gamma}\cdot\sum_{i=1}^N\left(\frac{\partial\hat{S}( \underline{z})}{\partial z_i}+\frac{\partial}{\partial z_i}\right)\frac{\partial}{\partial\eta_i} \\
\Psi_S&:= \left.\Psi_{\BP^n,\sigma, \tilde G_{\ud t}}\right|_{\ud t=\ud 1}=\Psi_{\BP^n}\circ E_{\tilde S}(\underline{z})=\Psi_{\BP^n}\circ e^\G.
\end{split}
\end{align*}

By \cite[Lemma 3.1]{PP}, we have
\begin{align*}
\tilde K_S (\lambda)=\left.e^{-\tilde \Gamma_{\ud t}} \circ \tilde \Delta \circ e^{\tilde \Gamma_{\ud t}}\right|_{\ud t=\ud 1} (\lambda)
=\tilde \Delta (\lambda)+\sum_{n\geq2}\frac{1}{(n-1)!}\ell^{\tilde \Delta}_n(\tilde \Gamma,\cdots,\tilde \Gamma,\lambda), \quad \lambda \in \widetilde \cAb
\end{align*}
where we recall from \eqref{onef}:
\begin{align*}
\tilde\Gamma|_{\ud t=1}=\hat S:=
\sum_{i,\underline{u}}\sum_{\ell\geq0}\gamma_\ell \left(\tilde{g}_{i,\underline{u}}y_i\underline{x}^{\underline{u}}\right)^{p^\ell}
\quad (\g_\ell:=\sum_{i=0}^\ell \frac{\g^{p^i}}{p^i} ).
\end{align*}
%
Since $\ell_m^{\tilde \Delta}=0$ for $m\geq 3$, the first part of Theorem \ref{stheorem} follows. Moreover, one can easily check that
$$
 \ell_2^{\tilde \Delta}(\hat S, \lambda)
=\left(\frac{\pi^{Mb}}{\gamma}\cdot\sum_{i=1}^N\frac{\partial\hat{S}( \underline{z})}{\partial z_i}\frac{\partial}{\partial\eta_i}\right)(\lambda), \quad\lambda \in \widetilde \cAb.
 $$
By \cite[Lemma 3.3]{PP}, we have
\begin{align*}
\Psi_{S}(\lambda)
&=\left.e^{-\sigma(\tilde \Gamma_{\ud t})} \circ \Psi_{\BP^n} \circ e^{\tilde \Gamma_{\ud t}}\right|_{\ud t=\ud 1} (\lambda)\\
&=\Psi_{\BP^n}\left(E_{\tilde S}(\underline{z})\cdot \lambda\right)
=\Psi_{\BP^n}\left(e^\G\cdot \lambda\right) \\
&=\left(\phi^{\Psi_{\BP^n}}_1(\lambda)+\sum_{n\geq2}\frac{1}{(n-1)!}\phi^{\Psi_{\BP^n}}_n(\Gamma,\cdots, \Gamma,\lambda)\right)\cdot e^{\Phi^{\Psi_{\BP^n}}(\Gamma)},\quad \lambda \in \widetilde \cAb.
\end{align*}

%
%

We want to expand this as
\begin{align*}
\Psi_S(\lambda)=\Psi_0(\Gamma)+\Psi_1(\Gamma)+\Psi_2(\Gamma)+\Psi_3(\Gamma)+\cdots,
\end{align*}
where $\Psi_n(c  \Gamma)=c^n\Psi_n(\Gamma)$ for $c\in\Bbbk$. 
By the definition of the Bell polynomials, we have
\begin{align*}
e^{\Phi^{\Psi_{\BP^n}}(\Gamma)}=\exp\left(\sum_{n\geq1}\frac{1}{n!}\phi_n^{\Psi_{\BP^n}}(\Gamma, \cdots, \Gamma)\right)=1+\sum_{n\geq1}\frac{1}{n!}B_n(\phi^{\Psi_{\BP^n}}_1(\Gamma),\cdots,\phi^{\Psi_{\BP^n}}_n(\Gamma,\cdots, \Gamma)).
\end{align*}
This expansion has an advantage that
\begin{align*}
B_n\left(\phi^{\Psi_{\BP^n}}_1(c\cdot\Gamma),\cdots,\phi^{\Psi_{\BP^n}}_n(c\cdot\Gamma,\cdots, c \cdot \Gamma)\right)=
c^n\cdot B_n\left(\phi^{\Psi_{\BP^n}}_1(\Gamma),\cdots,\phi^{\Psi_{\BP^n}}_n(\Gamma, \cdots, \Gamma)\right).
\end{align*}
for $c\in\Bbbk$. Therefore we get the expansion of desired form:
\begin{align*}
\Psi_S(\lambda)&=\Psi_{\BP^n}(\lambda)+\sum_{m\geq1}\sum_{\substack{j+k=m \\ j,k\geq0}}\frac{1}{j!k!} B_j(\phi^{\Psi_{\BP^n}}_1(\Gamma),\cdots,\phi^{\Psi_{\BP^n}}_j(\Gamma, \cdots, \Gamma))\cdot\phi^{\Psi_{\BP^n}}_{k+1}(\Gamma, \cdots, \Gamma,\lambda).
\end{align*}
\end{proof}

\begin{remark}
A merit of this formalism is that any smooth complete intersection $X=X_{\underline{G}}\subseteq\mathbf{P}^n_{\mathbb{F}_q}$ can be regarded as being deformed from the projective space $\mathbf{P}^n_{\mathbb{F}_q}$. Hence we may prove many properties of $(\widetilde \cAb, \tilde K_S)$ by checking it on the corresponding properties on the projective space and transporting it to $X$ via the deformation. It would be a good project to see whether this formula for $\Psi_S$ is actually helpful for an algorithmic computation of the zeta function.
\end{remark}


\begin{thebibliography}{99}

\bibitem{AdSp06} Adolphson, Alan; Sperber, Steven, On the Jacobian ring of a complete intersection, J. Algebra 304 (2006), no. 2, 1193--1227.

\bibitem{AdSp08} Adolphson, Alan; Sperber, Steven, On the zeta function of a projective complete intersection, Illinois J. Math. 52 (2008), no. 2, 389--417.


%
%
\bibitem{De74}
Deligne, Pierre, La conjecture de Weil. I. (French) Inst. Hautes \'Etudes Sci. Publ. Math. No. 43 (1974), 273--307.

\bibitem{Dim95}
A. Dimca, {Residues and cohomology of complete intersections}, Duke Math. J., 78 (1995) No. 1 89--100.

\bibitem{Dw62} 
Dwork, Bernard, On the zeta function of a hypersurface,  Inst. Hautes
\'{E}tudes Sci. Publ. Math. 12(1962), 5--68.

\bibitem{Dw64} 
Dwork, Bernard, On the zeta function of a hypersurface, II, Ann. of Math. (2) 80 (1964), 227--299.

\bibitem{Gr69} Griffiths, Phillip A.,
{ On the periods of certain rational integrals. I, II},
Ann. of Math. (2) 90 (1969), 460--495; ibid. (2) 90 (1969), 496--541. 

\bibitem{Ko91}
K. Konno,  {On the variational Torelli problem for complete intersections}, Comp. Math., 78 (1991), 271-296
\bibitem{KPP21} Kim, Yesule; Park, Jeehoon; Park, Junyeong.
Polynomial realizations of period matrices of projective smooth complete intersections and their deformation, https://arxiv.org/abs/2101.03488


\bibitem{LV12} Loday, Jean-Louis; Vallette, Bruno, {Algebraic operads.} Grundlehren der Mathematischen Wissenschaften, 346. Springer, Heidelberg, 2012. xxiv+634 pp.

\bibitem{Ma73}
 Mazur, B., Frobenius and the Hodge filtration (estimates), Ann. of Math. (2) 98 (1973), 58--95. 


\bibitem{Mnev} Mnev, P., Quantum field theory: Batalin-Vilkovisky formalism and its applications, University Lecture Series, Volume 72, 2019, 192 pp, Print ISBN: 978-1-4704-5271-1.

%
%


\bibitem{PP}
Park, Jae-Suk; Park, Jeehoon;
{Enhanced homotopy theory for period integrals of smooth projective hypersurfaces},
Communications in Number theory and Physics, Volume 10 (2016), Number 2, Pages 235--337. 


%

\end{thebibliography}
\end{document}